\documentclass[11pt,reqno]{amsart}

\usepackage[margin=1in]{geometry}

\title[Supercritically  diffusive \MG]{On the supercritically diffusive magneto-geostrophic equations}
\date{\today}

\author{Susan Friedlander}
\address{Department of Mathematics,
University of Southern California, 3620 S.~Vermont Ave.,
Los Angeles, CA 90089} \email{\tt susanfri@usc.edu}

\author{Walter Rusin}
\address{Department of Mathematics,
University of Southern California, 3620 S.~Vermont Ave.,
Los Angeles, CA 90089} \email{\tt wrusin@usc.edu}

\author{Vlad Vicol}
\address{Department of Mathematics,
The University of Chicago, 5734 University Ave.,
Chicago, IL 60637} \email{\tt vicol@math.uchicago.edu}

\usepackage{amsfonts,amsmath,latexsym,amssymb,verbatim,amsbsy,times,color,amsthm}

\usepackage{hyperref}

\theoremstyle{plain}
\newtheorem{theorem}{Theorem}[section]
\newtheorem{definition}[theorem]{Definition}
\newtheorem{lemma}[theorem]{Lemma}
\newtheorem{proposition}[theorem]{Proposition}
\newtheorem{corollary}[theorem]{Corollary}

\theoremstyle{definition}
\newtheorem{remark}[theorem]{Remark}

\def\tilde{\widetilde}
\numberwithin{equation}{section}

\renewcommand\hat{\widehat}
\def\ZZ{{\mathbb Z}}

\def\RR{{\mathbb R}}
\def\TT{{\mathbb T}}
\def\PP{\mathcal P}
\def\QQ{\mathbb Q}
\def\OO{\mathcal O}

\def\Th{\Theta}
\def\Ub{\boldsymbol U}

\def\MG{$(MG)$\ }

\def\supp{\mathop {\rm supp} \nolimits}

\def\kkx{k_{1}}
\def\kky{k_{2}}
\def\kkz{k_{3}}

\def\kk{\boldsymbol k}
\def\xx{\boldsymbol x}

\def\MM{\boldsymbol M}

\def\eps{\varepsilon}

\begin{document}


\begin{abstract}
We address the well-posedness theory for the magento-geostrophic equation, namely an active scalar equation in which the divergence-free drift velocity is one derivative more singular than the active scalar. In the presence of supercritical fractional diffusion given by $(-\Delta)^{\gamma}$ with $0<\gamma<1$, we discover that for $\gamma > 1/2$ the equations are locally well-posed, while for $\gamma < 1/2$ they are ill-posed, in the sense that there is no Lipschitz solution map. The main reason for the striking loss of regularity when $\gamma$ goes below $1/2$ is that the constitutive law used to obtain the velocity from the active scalar is given by an unbounded Fourier multiplier which is both even and anisotropic. Lastly, we note that the anisotropy of the constitutive law for the velocity may be explored in order to obtain an improvement in the regularity of the solutions when the initial data and the force have thin Fourier support, i.e. they are supported on a plane in frequency space. In particular, for such well-prepared data one may prove the local existence and uniqueness of solutions for all values of $\gamma \in (0,1)$.
\end{abstract}


\subjclass[2010]{76D03, 35Q35, 76W05}
\keywords{Magneto-geostrophic models, supercritical diffusion, active scalar equations, Hadamard ill-posedness, thin Fourier support, anisotropic constitutive law, continued fractions.}

\maketitle

\setcounter{tocdepth}{1} 
\tableofcontents

\section{Introduction}\label{sec:intro}

The geodynamo is the process by which the Earth's magnetic field is created and sustained by the motion of the fluid core which is composed of a rapidly rotating, density stratified, electrically conducting fluid. Recently Moffat \cite{Moffatt08} proposed a model for the geodynamo which is a reduction of the full magnetohydrodynamic system of PDE. This model can be viewed as a nonlinear active scalar equation in three dimensions. An explicit operator $\MM[\Th]$ encodes the physics of the geodynamo and produces the divergence free drift velocity $\Ub(t,\xx)$ from the scalar ``buoyancy'' field $\Th(t,\xx)$. We call this active scalar equation the magnetogeostrophic equation $(MG)$. It has some features in common with the much studied two dimensional surface quasigeostrophic equation $(SQG)$ (see \cite{CaffarelliVasseur10,ChaeConstantinCordobaGancedoWu11,ChaeConstantinWu11,ConstantinWu99,CordobaCordoba04,Ju04,KiselevNazarovVolberg07, Rusin11, Wu04} and references therein). However the $(MG)$ equation has a number of novel and distinctive features due to the strongly singular and anisotropic nature of the operator $\MM$.

The {\em critically diffusive} $(MG)$ equation is defined by the following system
\begin{align}
  &\partial_t \Th + \Ub \cdot \nabla \Th  =S + \kappa \Delta \Th \label{eq:1}\\
  &\nabla \cdot \Ub = 0\label{eq:2}
\end{align}
where $\kappa \geq 0$ is a physical parameter, and $S(t,\xx)$ is a $C^\infty$-smooth, bounded in time source term. The velocity  $\Ub$ is divergence-free, and is obtained from $\Th$ via
\begin{align}
U_j =  M_j  \Th \label{eq:3},
\end{align}
for all $j \in \{1,2,3\}$, where the $M_j$ are Fourier multiplier operators with symbols given explicitly by
\begin{align}
& \hat{M}_1(\kk) = \frac{2\Omega \kky \kkz |\kk|^2 - (\beta^2/\eta) \kkx \kky^{2} \kkz}{4 \Omega^2 \kkz^{2} |\kk|^2 + (\beta^2/\eta)^2 \kky^{4}}\label{eq:M:1}\\
& \hat{M}_2(\kk) = \frac{-2\Omega \kkx \kkz |\kk|^2 - (\beta^2/\eta) \kky^{3}\kkz}{4 \Omega^2 \kkz^{2} |\kk|^2 + (\beta^2/\eta)^2 \kky^{4}}\label{eq:M:2}\\
& \hat{M}_3(\kk) = \frac{(\beta^2/\eta) \kky^{2}(\kkx^{2} + \kky^{2})}{4 \Omega^2 \kkz^{2} |\kk|^2 + (\beta^2/\eta)^2 \kky^{4}}\label{eq:M:3}
\end{align}where the Fourier variable $\kk \in \ZZ^d$ is such that $\kkz\neq 0$. Note that $\hat{M}_{j}(\kk)$ is not defined on $k_{3}=0$, since for the self-consistency of the model it is assumed that $\Th$ and $\Ub$ have zero {\em vertical mean}. It can be directly checked that $k_{j} \cdot \hat{M}_j (\kk) = 0$ for all $\kk \in \ZZ^{d}\setminus \{ k_{3}=0\}$, and hence the velocity field $\Ub$ given by \eqref{eq:3} is indeed divergence-free. The physical parameters of the geodynamo are the following: $\Omega$ is the rotation rate of the Earth, $\eta$ is the magnetic diffusivity of the fluid core, $\kappa$ is the thermal diffusivity, and $\beta$ is the strength of the steady, uniform mean part of the magnetic field in the fluid core. The perturbation magnetic field vector $\boldsymbol{b}(t,x)$ is computed from $\Th(t,\xx)$ via the operator
\begin{align}
b_{j} =  (\beta/\eta) (-\Delta)^{-1} \partial_{2} M_{j} \Th,\qquad \mbox{for all}\ j \in \{1,2,3\}.
\end{align}
We observe that the presence of the underlying magnetic field, reflected in the parameter $\beta^2/\eta$, produces a non-isotropic structure in the symbols $\hat{M}_j$.

It is important to note that although the symbols $\hat{M}_{i}$ are $0$-order homogenous under the isotropic scaling $\kk \to \lambda \kk$, due to their anisotropy the symbols $\hat{M}_i$ are {not} bounded functions of $\kk$. To see this, note that whereas in the region of Fourier space where $k_{1} \leq \max \{ k_{2}, k_{3}\}$ the $\hat{M}_i$ are bounded by a constant, uniformly in $|\kk|$, this is not the case on the ``curved'' frequency regions where $k_{3} = {\mathcal O}(1)$ and $k_{2} = {\mathcal O}(|k_{1}|^{r})$, with $0< r \leq 1/2$. In such regions the symbols are unbounded, since as $|k_{1}|\rightarrow \infty$ we have
\begin{align}
|\hat{M}_{1}(k_{1},|k_{1}|^{r},1)| \approx |k_{1}|^{r},\ |\hat{M}_{2}(k_{1},|k_{1}|^{r},1)| \approx |k_{1}|,\ |\hat{M}_{3}(k_{1},|k_{1}|^{r},1)| \approx |k_{1}|^{2r}. \label{eq:unbounded:symbol}
\end{align}
In fact, it may be shown that $|\hat{\MM}(\kk)| \leq C |\kk|$, where $C(\beta,\eta,\Omega)>0$ is a fixed constant, and this bound is sharp. The fact that symbols are at most linearly \emph{unbounded} in the whole of Fourier space implies that $\MM[\Th]$ is the \emph{derivative} of a singular integral operator acting on $\Th$ (see \cite{FriedlanderVicol11a} for details), as opposed to the case of the $(SQG)$ equation where $\Ub$ is the Riesz transform of $\Th$. 

In \cite{FriedlanderVicol11a} we studied a class of nonlinear active scalar equations of which the $(MG)$ system (\ref{eq:1})--(\ref{eq:M:3}) is a member, namely (\ref{eq:1}) with
\begin{align}
	U_j = \partial_j T_{ij}\Th
\end{align}
for all $j \in \{ 1,2,3\}$ where $T_{ij}$ is a $3\times3$ matrix of Calder\'on-Zygmund operators of convolution type such that $\partial_i\partial_j T_{ij}\phi=0$ for any smooth function $\phi$. Inspired by the proof of Caffarelli and Vasseur \cite{CaffarelliVasseur10} for the global well-posedness of the critically diffusive $(SQG)$ equations, we used DeGiorgi techniques to obtain the global well-posedness of the critically diffusive $(MG)$ equations, namely (\ref{eq:1})--(\ref{eq:M:3}). We remark that  for both the diffusive $(MG)$ and $(SQG)$ equations {\em criticality}
is defined with respect to the $L^\infty$ maximum principle. In \cite{FriedlanderVicol11c} we considered the non-diffusive version of the $(MG)$ equations, namely (\ref{eq:1})--(\ref{eq:M:3}) with $\kappa=0$. In contrast with the critically diffusive problem where the $(MG)$ equation is globally well-posed and the solutions are $C^\infty$ smooth for positive time, we proved that when $\kappa=0$ the equation is Hadamard \emph{ill-posed} in any Sobolev space\footnote{The non-diffusive $(MG)$ equations are however locally well-posed in spaces of real-analytic functions, cf.~\cite{FriedlanderVicol11c}.} 
 which embeds in $W^{1,4}$.

In this present paper we study the fractionally diffusive version of the $(MG)$ equation, namely
\begin{align}
	&\partial_t \Th + \Ub \cdot \nabla \Th = S - \kappa (-\Delta)^\gamma\Th \label{eq:FractionalMG:1}\\
	& \nabla \cdot \Ub =0, \; \Ub = \MM[\Th] \label{eq:FractionalMG:2} 
\end{align}
with $U_j = M_j[\Th]$ where the symbols $\hat{M}_j$ are given by \eqref{eq:M:1}--\eqref{eq:M:3} and the parameter $\gamma \in (0,1)$. The main observation is that the structure of $\hat{M}_j$ (anisotropy, unboundedness, and evenness) and the {\em weak} smoothing effects of the nonlocal operator $(-\Delta)^\gamma$ combine to produce a sharp dichotomy across the value $\gamma=1/2$. More precisely, if $\gamma > 1/2$ the equations are locally well-posed, while if $\gamma < 1/2$ they are ill-posed in Sobolev spaces, in the sense of Hadamard. 

In section \ref{sec:>1/2} we use energy estimates to prove that for $\gamma \in (1/2,1)$ the fractionally diffusive $(MG)$ equations are locally well-posed in Sobolev spaces $H^s(\TT^3)$ for $s > 5/2 + (1-2 \gamma)$. With an additional small data assumption we obtain a global well-posedness result. To see why one may not use energy estimates to obtain local well-posedness of the \MG equations when $\gamma<1/2$, we point out that in the $H^{s}$ energy estimate for \eqref{eq:FractionalMG:1} there are only two terms for which more than ``$s$ derivatives'' fall on a single function:
\begin{align}
T_{bad} = \int \Lambda^{s} \Ub \cdot \nabla \Th\; \Lambda^{s} \Th \qquad \mbox{and} \qquad T_{good} = \int \Ub \cdot \nabla \Lambda^{s} \Th\; \Lambda^{s} \Th 
\end{align}
where we denoted $\Lambda = (-\Delta)^{1/2}$. Since $\nabla \cdot \Ub = 0$, upon integrating by parts we have $T_{good} = 0$. On the other hand, the term $T_{bad}$ does not vanish in general. The diffusion competing with $T_{bad}$ is given by $\kappa \| \Lambda^{s+\gamma} \Th \|_{L^{2}}^{2}$. Two issues arise. First, since $\Ub \approx \Lambda \Th$ and $\gamma<1/2$, and one cannot control $\| \Lambda^{s-\gamma} \Ub \|_{L^{2}} $ with $\| \Lambda^{s+\gamma} \Th \|_{L^{2}}$, and therefore bounding $T_{bad}$ could only be achieved by exploring some extra cancellation. Second, since the symbols $\hat{M}_{j}$ are even the operator $\MM$ is {\em not} anti-symmetric, thus one {cannot} re-write $T_{bad}$ as $-\int \Lambda^{s} \Th \cdot \MM[\nabla \Th\; \Lambda^{s}\Th]$ (note that if one could do this, a suitable commutator estimate of Coifman-Meyer type would close the estimates at the level of Sobolev spaces, as in \cite{ChaeConstantinCordobaGancedoWu11,FriedlanderVicol11c}). Hence there is no a commutator structure in $T_{bad}$, and the estimates do not seem to close at the level of Sobolev spaces when $0< \gamma < 1/2$. 

In section \ref{sec:<1/2} we consider the case $\gamma \in (0,1/2)$ and prove that the solution map associated with the Cauchy problem is {\em not} Lipschitz continuous with respect to perturbations in the initial data in the topology of a certain Sobolev space $X$. Hence the Cauchy problem is ill-posed in the sense of Hadamard. This is achieved by considering a specific steady profile $\Th_0$, and constructing functions which are ``close'' to $\Th_{0}$, but which in arbitrarily short time deviate arbitrarily far from $\Th_{0}$, when measured in $X$. The arguments hinge on using techniques of continued fractions to exhibit an unstable eigenvalue for the linearized equation (cf. \cite{FriedlanderStraussVishik97,MeshalkinSinai61}). These techniques produce a \emph{lower} bound on the growth rate of eigenvalues to the linearized equations, and in the case where $\gamma \in (0,1/2)$ we prove that the magnitude of this lower bound can be made arbitrarily large. Once unstable eigenvalues of arbitrarily large part have been obtained for the linearized equations, the Lipschitz ill-posednedness of the full non-linear equations may be obtained using classical arguments (see, e.g.~\cite{FriedlanderVicol11c,GerardVaretDormy10,GuoNguyen10,Renardy09,Tao06}). We emphasize that the mechanism producing ill-posedness is not merely the order one derivative loss in the map $\Th \to \Ub$. Rather, it is the combination of the derivative loss with the anisotropy of the symbol $M$ and the fact that this symbol is \emph{even}. We note that the even nature of the symbol of $M$ plays a central role in the proof of the non-uniqueness given in \cite{Shvydkoy11} for the $L^\infty$-weak solutions of the non-diffusive $(MG)$ equations. 

In section \ref{sec:=1/2} we study the more subtle transitional case when $\gamma=1/2$. In this case, if the initial data is ``small'' with respect to $\kappa$, we use energy estimates to prove that there exists a unique global solution. In dramatic contrast, there exists initial data which is ``large'' with respect to $\kappa$, and for which the associated linear operator has arbitrarily large unstable eigenvalues, which may be use to prove that the equations are Hadamard ill-posed. Situations where the qualitative behavior of the solution depends crucially on the norm of initial data are also encountered for other evolution equations. For example, in the two-dimensional Keller-Segel model of chemotaxis, properties of the solution are strongly dependent on the total mass $m$ of cells. If $m < 8\pi$, a global bounded solution of the initial value problem exists and is dispersed for $t\to \infty$, while for $m > 8\pi$, blow-up in finite time occurs. The critical case $m=8\pi$ is by now understood and the result states that a global solution exists and possibly becomes unbounded as $t \to \infty$ (see~\cite{HillenPainter09} and references therein). We also mention that for the Rayleigh-Taylor and the Helmholtz problems it is ``geometric'' conditions, rather than merely the size of the data, which leads to ill-posedness (see, for instance \cite{Ebin88}).

In section \ref{sec:well-prepared} we take advantage of the anisotropy of the symbols \eqref{eq:M:1}--\eqref{eq:M:3} and observe an interesting phenomenon: when the initial data and the source term are ``well-prepared'' in an appropriate sense, it is possible to prove stronger regularity results for the ensuing solutions. More specifically, when the initial data and the force have ``thin'' Fourier support, i.e. they are supported purely on a plane $\mathcal{P}_q= \{ \kk=(k_1,k_2,k_3) \in \ZZ^3: k_2= q k_1\} $ in frequency space,  where $q$ is a fixed non-zero rational number, the operator $\MM$ behaves like an operator of order zero. The reason is that the Fourier symbols $\hat{M}_j$ are unbounded only on curved regions in frequency space (cf.~\eqref{eq:unbounded:symbol} above), but if $\kk \in \PP_q$ the $\hat{M}_j(\kk)$ are bounded by a constant depending on $q$. This observation, combined with the fact that when the data and source have frequency support on $\PP_q$, then so does the solution of the \MG equations at all later times, means that the smoothing properties of the fractional diffusion term are stronger than they are for the generic data situation. Hence it is possible to prove stronger regularity results, and in particular the local existence and uniqueness of Sobolev solutions holds {\em for all} values of $\gamma \in (0,1)$.  
We note that this is not in contradiction with our results proven in sections~\ref{sec:<1/2} and \ref{sec:=1/2}, since in order to obtain ill-posedness we need to send $q \to 0$ (so that we obtain eigenvalues of the linear operator with arbitrarily large real part), whereas in the case of well-prepared data the value of $q>0$ is {\em fixed}. Other uses of thin Fourier support for different problems can be found, e.g. in \cite{BernicotGermain09,Tao06}.

%
%
%
%
%
%
%
%
%
%
%

\section{Preliminaries}

This section contains a few auxiliary results used in the paper. In particular, we recall the, by now classical, product and commutator estimates, as well as the Sobolev embedding inequalities. Proofs of these results can be found for instance in \cite{KenigPonceVega91,Stein93,Tao06}. Let us denote $\Lambda = (-\Delta)^{1/2}$.

\begin{proposition}[{\bf Product estimate}]
\label{prop:calculus}
If $s > 0$, then 	for all $f,g \in H^s\cap L^\infty$, and we have the estimates 
	\begin{align}
		&\|\Lambda^s(fg)\|_{L^p} \leq C\left( \|f\|_{L^{p_1}} \|\Lambda^s g\|_{L^{p_2}} + \|\Lambda^sf\|_{L^{p_3}} \|g\|_{L^{p_4}}\right), 
	\end{align}
	where $1/p = {1}/{p_1}+{1}/{p_2}= {1}/{p_3}+ {1}/{p_4}$, and $p,p_{2},p_{3} \in(1,\infty)$. In particular
	\begin{align}
		&\|\Lambda^s(fg)\|_{L^2} \leq C\left( \|f\|_{L^{\infty}} \|\Lambda^s g\|_{L^{2}} + \|\Lambda^sf\|_{L^{2}} \|g\|_{L^{\infty}}\right). \label{eq:prop:product}
	\end{align}
\end{proposition}

In the case of a commutator we have the following estimate.

\begin{proposition}[{\bf Commutator estimate}]
	Suppose that $s >0$ and $p \in (1,\infty)$. If, $f,g \in \mathcal{S}$, then
	\begin{align}
		&\|\Lambda^s(fg) - f \Lambda^{s} g \|_{L^p} \leq C\left(\| \nabla f\|_{L^{p_1}}\|\Lambda^{s-1} g\|_{L^{p_2}} + \|\Lambda^s f\|_{L^{p_3}}\|g\|_{L^{p_4}} \right) \label{eq:prop:commutator}
	\end{align}
	where $s>0$, ${1}/{p} = {1}/{p_1}+{1}/{p_2}={1}/{p_3}+{1}/{p_4}$, and $p,p_{2},p_{3} \in(1,\infty)$. 
\end{proposition}

We shall use as well the following Sobolev inequality.

\begin{proposition}[{\bf Sobolev inequality}] \label{prop:Sobolev}
Suppose that $q > 1$, $p \in [q,\infty)$ and ${1}/{p}={1}/{q}-{s}/{d}$. Suppose that $\Lambda^s f \in L^q$, then $f \in L^p$ and there is a constant $C > 0$ such that 
	\begin{align}
		\|f\|_{L^p} \leq C \|\Lambda^s f\|_{L^q}.
	\end{align}
\end{proposition}

Throughout this paper we shall make use the following definition of solutions of \eqref{eq:1}--\eqref{eq:3}.
\begin{definition}[\bf Solution of the \MG\ equation]
 Let $s \geq 1/2$, $T>0$, $\kappa>0$, and $\gamma \in (0,1)$. Given $\Th_0 \in H^s(\TT^3)$, we call a function
 \begin{align}
\Th \in L^\infty(0,T;H^s(\TT^3)) \cap L^2(0,T; H^{s+\gamma}(\TT^3))
\end{align}
a {\em solution} of the Cauchy problem \eqref{eq:1}--\eqref{eq:3}, if 
\begin{align}
\int_{\TT^3} \left( \Th(t,\cdot) - \Th_0 \right)  \varphi d\xx  + \kappa \int_0^t  \int_{\TT^3} \Th (-\Delta)^\gamma \varphi d\xx ds - \int_0^t \int_{\TT^3} \Th\; \Ub \cdot \nabla \varphi d\xx ds = 0 
\end{align}
holds for all test functions $\varphi \in C_0^\infty(\TT^3)$, and all $t\in (0,T)$.
\end{definition}

\section{The regime $1/2<\gamma < 1$: well-posedness results} \label{sec:>1/2}

It has been shown in \cite{FriedlanderVicol11a} that the system (\ref{eq:1})--(\ref{eq:M:3}) is globally well-posed when $\gamma =1$. In this section we prove that in the range $\gamma \in (1/2,1)$ equations (\ref{eq:1})--(\ref{eq:M:3}) are locally well-posed in $H^s(\TT^3)$. With an additional assumption on the size of the initial data and the source term, we obtain a global well-posedness result.
\footnote{For our convenience, we choose to work in sub-critical Sobolev spaces. The following proofs can be also carried out in the setting of critical Besov spaces $B^s_{2,1}$ or in general $B^{s}_{p,q}$.} 
 
Regarding arbitrarily large initial data, we obtain the following result.
 
\begin{theorem}[{\bf Local existence}]\label{thm:121:local}
Let $\gamma \in (1/2,1)$, and fix $s>5/2 + (1-2\gamma)$. Assume that $\Th_0 \in H^s(\TT^3)$ and $S \in L^\infty(0,\infty;H^{s-\gamma}(\TT^3))$ have zero-mean on $\TT^{3}$. Then there exists a time $T>0$ and a unique smooth solution 
\begin{align} \Th \in L^\infty(0,T;H^s(\TT^3)) \cap L^2(0,T;H^{s+\gamma}(\TT^3))
\end{align} of the Cauchy problem \eqref{eq:FractionalMG:1}--\eqref{eq:FractionalMG:2}.
\end{theorem}
\begin{proof}
	 We multiply equation (\ref{eq:1}) by $\Lambda^{2s}\Th$, integrate by parts to obtain the following $H^s$-energy inequality
\begin{align}\label{eq:p1}
&\frac{1}{2}\frac{d}{dt}\|\Lambda^s\Th\|^2_{L^2} + \kappa\|\Lambda^{s+\gamma}\Th\|^2_{L^2} \leq \left| \int_{\RR^3} S \Lambda^{2s}\Th\right| + \left| \int_{\RR^3} \Ub \cdot \nabla \Th \Lambda^{2s}\Th\right|.
\end{align}	
We estimate the first term on the right side by
\begin{align}
	&\left| \int_{\RR^3} S \Lambda^{2s}\Th\right| \leq \|\Lambda^{s-\gamma}S\|_{L^2}\|\Lambda^{s+\gamma}\Th\|_{L^2} \leq \frac{1}{2\kappa} \|\Lambda^{s-\gamma}S\|_{L^2}^2 + \frac{\kappa}{2}\|\Lambda^{s+\gamma}\Th\|_{L^2}^2.
\end{align}
To handle the second term we proceed as follows. First note that
\begin{align}
	& \left| \int_{\RR^3} \Ub \cdot \nabla \Th \Lambda^{2s}\Th\right| = \left| \int_{\RR^3} \Lambda^{s-\gamma}(\Ub\cdot\nabla \Th)\Lambda^{s+\gamma}\Th\right| \leq \|\Lambda^{s-\gamma}(\Ub \cdot \nabla \Th)\|_{L^2}\|\Lambda^{s+\gamma}\Th\|_{L^2}.
\end{align}
The estimate of the product term follows from Proposition \ref{prop:calculus}. Hence, we have
\begin{align}\label{eq:p8}
	&\|\Lambda^{s-\gamma}(\Ub \cdot \nabla \Th)\|_{L^2} \leq C(\|\Lambda^{s-\gamma}\Ub\|_{L^p}\|\nabla \Th\|_{L^{\frac{2p}{p-2}}} + \|\Lambda^{s-\gamma}\nabla \Th\|_{L^p}\|\Ub\|_{L^{\frac{2p}{p-2}}})
\end{align}
for some $2 < p < \infty$ to be chosen later.
Recall that $U_j = M_j\Th $, where the symbol of the multiplier $M_j$ enjoys a uniform bound $|\hat{M}_j(\kk)| \leq C|\kk|$, for $j=1,2,3$. Therefore $\Lambda^{-1} M_{j}$ are bounded operators on $L^{p}$ (see \cite[Section 4]{FriedlanderVicol11a} for details) and we obtain
\begin{align}
	&\|\Lambda^{s-\gamma}\Ub\|_{L^p} \leq C \|\Lambda^{s-\gamma+1}\Th\|_{L^p}
\end{align}
and 
\begin{align} 
\| \Ub \|_{L^{\frac{2p}{p-2}}} \leq C \| \Lambda \Th \|_{L^{\frac{2p}{p-2}}}.
\end{align}
Hence, the product term is bounded by
\begin{align}\label{eq:2.8}
	&\|\Lambda^{s-\gamma}(\Ub \cdot \nabla \Th)\|_{L^2} \leq C \| \Lambda \Th\|_{L^{\frac{2p}{p-2}}} \|\Lambda^{s-\gamma+1}\Th\|_{L^p}.
\end{align}
We now fix an arbitrary $p$  such that
\begin{align*}
\frac{3}{s-1} < p  <  \frac{6}{2(1-2\gamma)+3} = \frac{6}{5-4\gamma}.
\end{align*}
Note that $p>2$ since $s>3/2$, and the range for $p$ is non-empty since $s> 5/2 + (1-2\gamma)$. For $\gamma \in (1/2,1)$, our choice of $p$ and the Sobolev embedding (Proposition~\ref{prop:Sobolev}) gives
\begin{align} 
\|\Lambda^{s-\gamma+1}\Th\|_{L^p} \leq \|\Lambda^s\Th\|^{1-\zeta}_{L^2}\|\Lambda^{s+\gamma}\Th\|_{L^2}^{\zeta}
\end{align}
where $\zeta \in (0,1)$ may be computed explicitly from $\gamma \zeta = 3/2 - 3/p + 1-\gamma$. In addition, since $\Th$ has zero mean, and $p> 3/(s-1)$, from the Sobolev embedding we obtain
\begin{align}\label{eq:2.9}
	&\|\Lambda \Th\|_{L^{\frac{2p}{p-2}}} \leq C \|\Lambda^s\Th\|_{L^2}.
\end{align}
Combining estimates (\ref{eq:p1})-(\ref{eq:2.9}) gives
\begin{align}\label{eq:p6}
	& \frac{d}{dt}\|\Lambda^s\Th\|^2_{L^2} + \kappa\|\Lambda^{s+\gamma}\Th\|^2_{L^2} \leq \frac{1}{\kappa} \|\Lambda^{s-\gamma}S\|_{L^2}^{2} + C\|\Lambda^s\Th\|_{L^2}^{2-\zeta}\|\Lambda^{s+\gamma}\Th\|_{L^2}^{1+\zeta}
\end{align}
where $ 0 < \zeta< 1$ is as defined earlier.
The second term on the right side of (\ref{eq:p6}) is bounded using the $\epsilon$-Young inequality as 
\begin{align}
 \frac{\kappa}{2}\|\Lambda^{s+\gamma}\Th\|_{L^2}^2 +C \kappa^{- \frac{2(1+\zeta)}{1-\zeta}} \|\Lambda^s\Th\|_{L^2}^{\frac{2 (2-\zeta)}{1-\zeta}}.
\end{align}
and we finally obtain the following estimate
\begin{align}\label{eq:p7}
	&\frac{d}{dt}\|\Lambda^s\Th\|^2_{L^2} + \frac{\kappa}{2}\|\Lambda^{s+\gamma}\Th\|^2_{L^2} \leq \frac{1}{\kappa}\|\Lambda^{s-\gamma}S\|_{L^2}^2 + C \kappa^{- \frac{2(1+\zeta)}{1-\zeta}} \|\Lambda^s\Th\|_{L^2}^{\frac{2 (2-\zeta)}{1-\zeta}}.
\end{align}
Using Gronwall's inequality, from estimate (\ref{eq:p7}) we may deduce the existence of a positive time $T=T(\|S\|_{L^\infty(0,T;H^{s-\gamma})},\|\Th_0\|_{H^s},\kappa)$ such that $\theta \in L^\infty(0,T;H^s(\TT^3))\cap L^2(0,T;H^{s+\gamma}(\TT^3))$. Note that we have $2(2-\zeta)/(1-\zeta) >2$, and hence we may not obtain the global existence of solutions from the energy estimate \eqref{eq:p7}, if the initial data has large $H^{s}$ norm. These a priori estimates can be made formal using a standard approximation procedure. We omit further details.
\end{proof}

The second main result of this section concerns global well-posedness in case of small initial data.

\begin{theorem}[{\bf Global existence for small data}]\label{thm:121:global}
Let $\gamma$ and $S$ be as in the statement of Theorem~\ref{thm:121:local}, and let $ \Th_0 \in H^{s}(\TT^3)$ have zero-mean on $\TT^3$, where $s > 5/2+(1-2\gamma)$. There exists a  small enough constant $\varepsilon>0$ depending on $\kappa$, such that if  $\Vert\Th_0\Vert_{L^2}^\alpha \Vert \Th_{0} \Vert_{H^{s}}^{1-\alpha} +\Vert\Th_0\Vert_{L^2}^\alpha \Vert S \Vert_{L^\infty(0,\infty;H^{s-\gamma})}^{1-\alpha}\leq \varepsilon$, where $\alpha = 1 - (5/2+1-2\gamma)/s$, then the unique smooth solution $\Th$ of the Cauchy problem  \eqref{eq:FractionalMG:1}--\eqref{eq:FractionalMG:2} is global in time, i.e.
\begin{align} \Th \in L^\infty(0,\infty;H^s(\TT^3)). 
\end{align}
\end{theorem}
\begin{proof}
We proceed as in the proof of Theorem \ref{thm:121:local}. The product term in (\ref{eq:p8}) is now estimated by
\begin{align}
	&\|\Lambda^{s-\gamma}(\Ub \cdot \nabla \Th)\|_{L^2} \leq C(\|\Lambda^{s-\gamma}\Ub\|_{L^p}\|\nabla \Th\|_{L^{\frac{2p}{p-2}}}+ \|\Lambda^{s-\gamma}\nabla \Th\|_{L^p}\|\Ub\|_{L^\frac{2p}{p-2}}),
\end{align}
where
\begin{align*}
	p = \frac{6}{2(1-2\gamma)+3}
\end{align*}
so that
\begin{align}
	&\|\Lambda^{s+1-\gamma}\Th\|_{L^p} \leq C\|\Lambda^{s+\gamma}\Th\|_{L^2}.
\end{align}
With this choice of $p$ and the above embedding, the product estimate gives us
\begin{align}\label{eq:p9}
	\|\Lambda^{s-\gamma}(\Ub \cdot \nabla \Th)\|_{L^2} \leq C\|\nabla \Th\|_{L^{\frac{2p}{p-2}}}\|\Lambda^{s+\gamma} \Th\|_{L^2}.
\end{align}
Combining (\ref{eq:p9}) with (\ref{eq:p1}) and proceeding as in (\ref{eq:p7}) we obtain
\begin{align}\label{eq:p10}
	&\frac{1}{2}\frac{d}{dt}\|\Lambda^s\Th\|^2_{L^2} + \kappa\|\Lambda^{s+\gamma}\Th\|^2_{L^2} \leq  \|\Lambda^{s-\gamma}S\|_{L^2}\|\Lambda^{s+\gamma}\Th\|_{L^2} + C\|\nabla \Th\|_{L^{\frac{2p}{p-2}}}\|\Lambda^{s+\gamma} \Th\|_{L^2}^2,
\end{align}
which in turn implies
\begin{align}\label{eq:p11}
	&\frac{1}{2}\frac{d}{dt}\|\Lambda^s\Th\|^2_{L^2} + \frac{\kappa}{2}\|\Lambda^{s+\gamma}\Th\|^2_{L^2} \leq \frac{1}{2\kappa}\|\Lambda^{s-\gamma}S\|_{L^2}^2 + C\|\nabla \Th\|_{L^{\frac{2p}{p-2}}}\|\Lambda^{s+\gamma} \Th\|_{L^2}^2.
\end{align}
Observe that
\begin{align}\label{eq:p12}
	\|\nabla \Th\|_{L^{2p/(p-2)}} \leq C \|\Th\|_{L^2}^{\alpha}\|\Lambda^s \Th\|_{L^2}^{1-\alpha},
\end{align}
where $\alpha = 1- (5/2+1-2\gamma)/s$. 
Therefore, if
\begin{align}\label{eq:p13}
	\|\Th\|_{L^2}^{\alpha}\|\Lambda^s \Th\|_{L^2}^{1-\alpha} \leq \frac{\kappa}{4C},
\end{align}
estimate (\ref{eq:p11}), combined with the Poincar\'e inequality $\| \Lambda^{s+\gamma} \Th \|_{L^{2}} \geq \| \Lambda^{s} \Th \|_{L^{2}}$, shows that
\begin{align}
\frac{d}{dt} \|\Lambda^s\Th\|^2_{L^2} + \frac{\kappa}{2}\|\Lambda^{s}\Th\|^2_{L^2} \leq \frac{1}{\kappa} \| S\|_{L^{\infty}_{t}H^{s-\gamma}_{x}}^{2} .
\end{align}
and hence
\begin{align} \label{eq:p13a} 
 \|\Lambda^s\Th(t,\cdot) \|^2_{L^2}  \leq \|\Lambda^s\Th_{0} \|^2_{L^2} + \frac{2}{\kappa^{2}} \| S\|_{L^{\infty}_{t}H^{s-\gamma}_{x}}^{2}
\end{align}
for all $t>0$. Note that taking the $L^2$-product of the equation with $\Th$ gives for any $t >0$
\begin{align}\label{eq:p14}
	\|\Th(t,\cdot)\|^2_{L^2} + \kappa \int_0^t\|\Lambda^\gamma \Th\|^2_{L^2}\;d\tau \leq \|\Th_0\|^2_{L^2},
\end{align}
which gives us a basic uniform estimate of $\Th$ in $L^{\infty}_{t}L^2_{x}$. Hence, from \eqref{eq:p13a} and \eqref{eq:p14} we obtain that condition \eqref{eq:p13} is satisfied for all $t>0$ as long as we have
\begin{align}\label{eq:p15}
 	& \|\Th_0\|_{L^2}^{\alpha}\|\Lambda^s \Th_0\|_{L^2}^{1-\alpha} + \|\Th_0\|_{L^2}^{\alpha} \|\Lambda^{s-\gamma}S(\tau,\cdot)\|_{L^\infty(0,\infty;L^2)}^{1-\alpha} < \epsilon,
\end{align}
where $\epsilon$ is sufficiently small, thereby concluding the proof of the theorem.
\end{proof}



\section{The regime $0<\gamma < 1/2$: ill-posedness results} \label{sec:<1/2}

We now turn to the situation when $\gamma \in (0,1/2)$, $\kappa > 0$, and prove that there exists an example of an initial datum and time independent force for which it is possible to prove that the fractionally diffusive \MG equation is ill-posed in the sense of Hadamard. The arguments by which we prove this follow the lines of the proof in the case when $\kappa = 0$ in \cite{FriedlanderVicol11c}. We sketch here the proof of $\gamma \in (0,1/2)$, but as we shall prove below in Section~\ref{sec:=1/2_large_data}, the result is also true when $\gamma = 1/2$ and the initial data is {\em large} in a suitable sense.

\subsection{Linear ill-posedness in $L^2$} \label{eq:sec:linear:ill}
The classical approach to Hadamard ill-posedness for non-linear problems is to first linearize the equations about a suitable steady state $\Theta_0$, and then prove that the resulting linear operator $L$ has eigenvalues with arbitrarily large real part in the unstable region (see for instance~\cite{FriedlanderVicol11c,GerardVaretDormy10,GuoNguyen10,Renardy09}, and references therein). 
Let $\Theta_{0} = a \sin (mx_{3})$ for some positive amplitude $a$ and integer frequency $m \geq 1$,  and define $S = \kappa a m^{2\gamma} \sin (m x_{3})$. It is not hard to verify that $\Theta_{0}$ is indeed a steady state of \eqref{eq:1}--\eqref{eq:2}. We consider the {\em linear} evolution of the perturbation $\theta = \Th - \Th_{0}$, obtained from linearizing \eqref{eq:1}--\eqref{eq:2}
\begin{align} 
\partial_{t} \theta + a m \cos(m x_{3}) M_{3}\theta+ \kappa (-\Delta)^{\gamma} \theta =  0 \label{eq:SN:5.1}, 
\end{align}
where we recall that $M_{3}$ is the Fourier multiplier with symbol defined in \eqref{eq:M:3}.
The following theorem states that the linear equation \eqref{eq:SN:5.1} is Hadamard Lipschitz ill-posed in $L^{2}$.
\begin{theorem}[\bf Linear ill-posedness] \label{thm:SN:5.1}
The Cauchy problem associated to the linear evolution
\begin{align}
  \partial_t \theta = L\theta \label{eq:linearequation}
\end{align} where the linear operator $L$ and the steady state $\Theta_{0}$ are given by
\begin{align}
L\theta(\xx,t) &= - M_3 \theta (\xx,t)\; \partial_3 \Theta_0(x_3) - \kappa (-\Delta)^{\gamma} \theta\label{eq:L:def}\\
\Theta_0(x_3) & = a \sin(m x_3) \label{eq:steadystatedef}
\end{align}for some $a>0$,  integer $m\geq 1$, and $\gamma \in (0,1/2]$, is ill-posed in the sense of Hadamard over $L^{2}$. More precisely, for any $T>0$ and any $K>0$, there exists a real-analytic initial data $\theta(0,\xx)$ such that the Cauchy problem associated to \eqref{eq:linearequation}--\eqref{eq:steadystatedef} has no solution $\theta \in L^{\infty}(0,T;L^{2})$ satisfying
\begin{align}
\sup_{t\in [0,T)} \Vert \theta(t,\cdot) \Vert_{L^{2}} \leq K \Vert \theta(0,\cdot)\Vert_{Y} \label{eq:thm:boundedness}
\end{align}
where $Y$ is any Sobolev space embedded in $L^{2}$.
\end{theorem}
As we mentioned above, in order to prove Theorem~\ref{thm:SN:5.1} we construct a sequence of eigenvalues $\sigma^{(j)}$ of the operator $L$, which diverge to $\infty$ as $j \rightarrow \infty$.
For this purpose, given any {\em fixed} integer $j \geq 1$ we seek a solution $\theta$ to \eqref{eq:linearequation} of the form
\begin{align} 
\theta(t,\xx) = e^{\sigma t} \sin(j^{2} x_{1}) \sin(j x_{2}) \sum_{p\geq 1} c_{p} \sin(m p x_{3}) \label{eq:SN:5.2}
\end{align}
where the sequence $c_{p}$ decays rapidly as $p \rightarrow \infty$. We shall construct such a solution $\theta$ with $\sigma$ real and positive, and in addition obtain bounds on the value of $\sigma$. Inserting \eqref{eq:SN:5.2} into \eqref{eq:SN:5.1} and using the definition of $M_{3}$, we obtain
\begin{align} 
\sigma_p \sum_{p\geq 1} c_{p} \sin(m p x_{3}) + \sum_{p\geq 1} \frac{c_{p}}{\alpha_{p}} \left( \sin ( m (p+1) x_{3}) + \sin ( m (p-1) x_{3}) \right) = 0, \label{eq:SN:5.3}
\end{align}
where we have denoted 
\begin{align} 
\sigma_{p} = \sigma + \kappa (j^{4} + j^{2} + (mp)^{2})^{\gamma} \label{eq:SN:5.6}
\end{align}
and
\begin{align} 
\alpha_{p} = \frac{2^{3} \Omega^{2} (mp)^{2} (j^{4} + j^{2} + (mp)^{2}) + 2 \mu^{2} j^{4}}{a \mu m j^{2} (j^{4} + j^{2})} \label{eq:SN:5.4}
\end{align}
for any integer $p\geq 1$ (note that $j$ is fixed). Here  $\mu = \beta^{2}/\eta$. An essential feature of the coefficients $\alpha_{p}$ is that they grow rapidly with $p$ as $p\rightarrow \infty$. Equation \eqref{eq:SN:5.4} gives the recurrence relation for the unknown coefficients $c_{p}$:
\begin{align} 
&\sigma_{p} c_{p} + \frac{c_{p+1}}{\alpha_{p+1}} + \frac{c_{p-1}}{\alpha_{p-1}} = 0, \mbox{ for } p\geq 2 \label{eq:SN:5.7}\\
&\sigma_{1} c_{1} + \frac{c_{2}}{\alpha_{2}} = 0, \mbox{ for } p = 1. \label{eq:SN:5.8}
\end{align}
Note that given any $\sigma>0$ (which then uniquely defines $\sigma_{p}$ for all $p\geq 1$ cf.~\eqref{eq:SN:5.6}) and given any $c_{1} > 0$, one may use the recursion relations \eqref{eq:SN:5.7}--\eqref{eq:SN:5.8} to solve for all $c_{p}$ with $p \geq 2$. However, only for suitable values of $\sigma$ do the $c_{p}$'s vanish sufficiently fast so that $\theta$ is $C^{\infty}$ smooth. In this direction we prove:
\begin{lemma} \label{lemma:SN:5.2}
 Let $\alpha_{p}$ be defined by \eqref{eq:SN:5.4}, where the positive integers $m$ and $j$ are fixed, and $\kappa, \mu, \Omega, a$ are fixed physical parameters. Also, define $c_{1} = \alpha_{1}$. Then there exists a real eigenvalue $\sigma = \sigma^{(j)}>0$, and rapidly decaying sequence $\{c_{p}\}_{p\geq 2}$ which satisfies \eqref{eq:SN:5.7}--\eqref{eq:SN:5.8}. Furthermore the lower bound 
 \begin{align} 
\sigma^{(j)} > \frac{a \mu m j^{2} (j^{4} + j^{2})}{2^{3} \Omega^{2} m^{2} (j^{4} + j^{2} + 4 m^{2}) + 2 \mu^{2} j^{4} } - \kappa (j^{4} + j^{2} + 4 m^{2})^{\gamma}
\end{align}
holds, and $c_{p} = \OO(C^{p}/ (p-1)!^{4})$ as $p \rightarrow \infty$, for some constant $C$ which depends on $\sigma^{(j)},j$ and the physical parameters.
\end{lemma}
\begin{proof}[Proof of Lemma~\ref{lemma:SN:5.2}]
 We define
 \begin{align} 
\eta_{p} = \frac{c_{p} \alpha_{p-1}}{c_{p-1} \alpha_{p}} \label{eq:SN:5.9}
\end{align}
and write \eqref{eq:SN:5.7}--\eqref{eq:SN:5.8} in the form
\begin{align} 
& \sigma_{p} \alpha_{p} + \eta_{p+1} + \frac{1}{\eta_{p}} = 0, \qquad p \geq 2\label{eq:SN:5.10}\\
& \sigma_{1} \alpha_{1} + \eta_{2} =0, \qquad p=1 \label{eq:SN:5.11}.
\end{align}
Using \eqref{eq:SN:5.10} to write $\eta_{2}$ as an infinite continued fraction and equating with $\eta_{2}$ given by \eqref{eq:SN:5.11} gives the characteristic equation
\begin{align} 
\sigma_{1} \alpha_{1} = \frac{1}{ \sigma_{2} \alpha_{2} - \frac{1}{ \sigma_{3} \alpha_{3} - \frac{1}{ \sigma_{4} \alpha_{4} - {\dots}} } }  \label{eq:SN:5.12}.
\end{align}
Recalling the definition of $\sigma_{p}$ for $p\geq 1$ cf.~\eqref{eq:SN:5.6}, we observe that \eqref{eq:SN:5.12} is an equation with only one unknown, namely $\sigma$. 
For real values of $\sigma$ we define the infinite continued fraction $F_{p}(\sigma)$ by
\begin{align} 
F_{p}(\sigma) = \frac{1}{\sigma_{p} \alpha_{p} - \frac{1}{\sigma_{p+1} \alpha_{p+1} - \frac{1}{\sigma_{p+2} \alpha_{p+2} - \dots}}} \label{eq:SN:5.13}
\end{align}
and the 
function
\begin{align} 
G_{p}(\sigma) = \frac{\sigma_{p} \alpha_{p} - \sqrt{\sigma_{p}^{2} \alpha_{p}^{2} - 4}}{2} = \frac{2}{\sigma_{p} \alpha_{p} + \sqrt{\sigma_{p}^{2} \alpha_{p}^{2} -4 }} \label{eq:SN:5.14}
\end{align}
for all $p\geq 2$, and all $\sigma$ such that $\sigma_{2} \alpha_{2}> 2$. We note that due to the very rapid growth of the $\alpha_{p}$'s, for real $\sigma$ the continued fraction $F_{p}$ defines a function which is smooth except for a set of points such with $\sigma_{2} \alpha_{2} <2$. For the rest of the proof we will always assume that $\sigma$ is real such that $\sigma_{2} \alpha_{2} >  2$, which also implies that $\sigma_{p} \alpha_{p}>2$ for all $p\geq 2$.

We note that by construction $G_{p}$ satisfies
\begin{align} 
G_{p}(\sigma) = \frac{1}{\sigma_{p} \alpha_{p} - G_{p}(\sigma)} = \frac{1}{\sigma_{p} \alpha_{p} - \frac{1}{\sigma_{p} \alpha_{p} - \frac{1}{\sigma_{p} \alpha_{p} - \dots}}}. \label{eq:SN:5.15}
\end{align}
Since we have $\sigma_{p} \alpha_{p} \rightarrow \infty$ as $p \rightarrow \infty$, we have that $G_{p}(\sigma) \rightarrow 0$ as $p \rightarrow \infty$ for every fixed $\sigma$.
Also, we clearly have
\begin{align} 
G_{2}(\sigma) > G_{3}(\sigma) > G_{4}(\sigma) > \ldots \geq 0 \label{eq:SN:5.16}
\end{align}
and
\begin{align} 
G_{p+1}(\sigma) < G_{p}(\sigma) < \sigma_{p} \alpha_{p} \label{eq:SN:5.17}
\end{align}
for all $p\geq 2$. Hence, from \eqref{eq:SN:5.15} and \eqref{eq:SN:5.16} we obtain
\begin{align} 
G_{2}(\sigma) > \frac{1}{\sigma_{2} \alpha_{2} - G_{3}(\sigma)} > \frac{1}{\sigma_{2} \alpha_{2} - \frac{1}{\sigma_{3} \alpha_{3} - G_{4}(\sigma)}} > 0 \label{eq:SN:5.18}.
\end{align}
An inductive argument then gives
\begin{align} 
G_{2}(\sigma) >\frac{1}{\sigma_{2} \alpha_{2} - \frac{1}{\sigma_{3} \alpha_{3} - \frac{1}{\sigma_{4} \alpha_{4} - \dots} }} = F_{2}(\sigma) \geq 0 \label{eq:SN:5.19}
\end{align}
for all $\sigma$ such that $\sigma_{2} \alpha_{2} > 2$. Repeating this constructive argument we obtain that $G_{3}(\sigma) > F_{3}(\sigma) \geq 0$, and hence 
\begin{align} 
0< \sigma_{2} \alpha_{2} - G_{3}(\sigma) < \sigma_{2} \alpha_{2} - F_{3}(\sigma) < \sigma_{2} \alpha_{2}
\end{align}
so that 
\begin{align} 
F_{2}(\sigma)  = \frac{1}{\sigma_{2} \alpha_{2} - F_{3}(\sigma)} > \frac{1}{\alpha_{2} \sigma_{2}} \label{eq:SN:5.20}.
\end{align}

In summary, we have proven that $1/(\sigma_{2} \alpha_{2}) < F_{2}(\sigma) < G_{2}(\sigma)$ for all $\sigma$ such that $\sigma_{2} \alpha_{2} > 2$, and the straight line $\sigma_{1}\alpha_{1}$ intersects both the graph of $1/(\sigma_{2} \alpha_{2})$ and the graph of $G_{2}(\sigma)$ in this rage of $\sigma$. Hence, it follows that there exists a real solution $\sigma^{(j)}$ of the equation
\begin{align} 
F_{2}(\sigma^{(j)}) = \sigma_{1}^{(j)} \alpha_{1} \label{eq:SN:5.21},
\end{align}
with $\sigma_{2}^{(j)} \alpha_{2} > 2$, where we denote $\sigma_{p}^{(j)} = \sigma^{(j)} + \kappa (j^{4} + j^{2} + m^{2} p^{2})^{\gamma}$, for all $p\geq 2$. That is, $\sigma^{(j)}$ is a real positive solution of the characteristic equation \eqref{eq:SN:5.12}. Furthermore, due to \eqref{eq:SN:5.19} and \eqref{eq:SN:5.20} after a short calculation we obtain an upper and a lower bound on $\sigma^{(j)}$, namely
\begin{align} 
\frac{1}{\alpha_{1} \alpha_{2}} < \sigma_{1}^{(j)} \sigma_{2}^{(j)} < \frac{2}{\alpha_{1} \alpha_{2}} \label{eq:SN:5.23}.
\end{align}

For $\sigma^{(j)}$ satisfying \eqref{eq:SN:5.21}, we now construct the sequence $c_{p}$ which decays rapidly as $p \rightarrow \infty$. The recursion relation \eqref{eq:SN:5.10}--\eqref{eq:SN:5.11} uniquely defines the values of $\eta_{p}$. After letting $c_{1} = \alpha_{1}$, we define
\begin{align} 
c_{p} = \alpha_{p} \eta_{p} \eta_{p-1} \ldots \eta_{2} \label{eq:SN:5.24}
\end{align}
for all $p \geq 2$. This sequence satisfies \eqref{eq:SN:5.7}--\eqref{eq:SN:5.8} by construction. Furthermore, we observe that $\eta_{p} = - F_{p}(\sigma^{(j)})$. Repeating the arguments which gave \eqref{eq:SN:5.19} and \eqref{eq:SN:5.20} with $2$ replaced by $p$ gives
\begin{align} 
\frac{-2}{\sigma_{p}^{(j)} \alpha_{p}} < \eta_{p} < \frac{-1}{\sigma_{p}^{(j)} \alpha_{p}}. \label{eq:SN:5.26}
\end{align}
Moreover, from \eqref{eq:SN:5.4} we have that $\alpha_{p} = \OO(p^{4})$ as $p\rightarrow \infty$, and hence from \eqref{eq:SN:5.26} we obtain that $\eta_{p} = \OO(p^{-4})$. Thus it follows from \eqref{eq:SN:5.24} that $c_{p} \rightarrow 0$ as $p\rightarrow \infty$, and this convergence is very fast, namely $\OO(C^{p} (p-1)!^{-4})$ as $p \rightarrow \infty$, for some positive constant $C = C(\sigma^{(j)},\mu,\Omega,a,m,j)$. Therefore the solution $\theta^{(j)}(\xx,t)$ given by \eqref{eq:SN:5.2} with $\sigma$ replaced by $\sigma^{(j)}$, and $c_{p}$ as defined by \eqref{eq:SN:5.24}, lies in any Sobolev space, it is $C^{\infty}$ smooth, and even real-analytic.

We substitute for $\alpha_{1}$ and $\alpha_{2}$ from \eqref{eq:SN:5.4} into the bound \eqref{eq:SN:5.23} to obtain estimates on $\sigma^{(j)}$ given by 
\begin{align} 
\frac{a \mu m j^{2}(j^{4}+j^{2})}{2^{3} \Omega^{2} m^{2} (j^{4}+j^{2} + 4 m^{2}) + 2 \mu^{2} j^{4} }  - \kappa (j^{4} + j^{2} + 4m^{2})^{\gamma} < \sigma^{(j)}  \label{eq:SN:5.27}
\end{align}
and 
\begin{align} 
\sigma^{(j)} < \frac{2 a \mu m j^{2}(j^{4}+j^{2})}{2^{3} \Omega^{2} m^{2} (j^{4}+j^{2} + m^{2}) + 2 \mu^{2} j^{4} } - \kappa (j^{4} + j^{2} + m^{2})^{\gamma} \label{eq:SN:5.28}.
\end{align}

We recall the role of the constants in the above expressions \eqref{eq:SN:5.27}--\eqref{eq:SN:5.28} for the bounds on $\sigma^{(j)}$. 
The physical parameters are $\kappa$ the coefficient of thermal diffusivity, $\Omega$ the rotation rate of the system, $\mu$ is related to the underlying magnetic field, and $a$ is the magnitude of the steady {\em buoyancy} $\Th_{0}$. Expressions  \eqref{eq:SN:5.27} and \eqref{eq:SN:5.28} indicate wide ranges of the physical parameters for which there exist eigenvalues $\sigma^{(j)}$ which are positive as postulated in the construction of the solution to \eqref{eq:SN:5.1}. We note that in the Earth's fluid core $\kappa$ is very small.
\end{proof}

\begin{lemma} \label{lemma:SN:5.3}
Let $\gamma \in (0,1/2)$. Fix $a,\mu,\Omega,\kappa>0$ and an integer $m\geq 1$. There exists an integer $j_{0} = j_{0}(a,m,\mu,\Omega,\kappa,\gamma)$ such that for all $j \geq j_{0}$ there exists a $C^{\infty}$ smooth initial datum $\theta^{(j)}(0,\xx)$ with $\Vert \theta^{(j)}(0,\cdot) \Vert_{L^{2}}=1$ and a $C^{\infty}$ smooth solution $\theta^{(j)}(t,\xx)$ of the initial value problem associated with the linearized \MG equation \eqref{eq:linearequation}, such that 
\begin{align} 
\Vert \theta^{(j)}(t,\cdot) \Vert_{L^{2}} \geq \exp\left( j^{2} C_{*} t\right)
\end{align}
for all $t>0$, where $C_{*} = C_{*}(a,m,\mu,\Omega,\kappa,j_{0})$ is a positive constant defined in \eqref{eq:SN:5.29} below.
\end{lemma}
\begin{proof}[Proof of Lemma~\ref{lemma:SN:5.3}]
For any $j \geq 1$,  Lemma~\ref{lemma:SN:5.2} guarantees the existence of an eigenvalue $\sigma^{(j)}$ of the linear \MG operator $L$, with associated $C^{\infty}$ smooth eigenfunction $\theta^{(j)}(0,\xx)$ given by letting $t=0$ in \eqref{eq:SN:5.2}. Then $\theta^{(j)}(t,\xx) = \exp(\sigma^{(j)} t ) \theta^{(j)}(0,\cdot)$ is a solution of \eqref{eq:linearequation}. It is clear that one may normalize $\theta^{(j)}(0,\cdot)$ to have 
$L^{2}$ norm equal to $1$, and hence $\Vert \theta^{(j)}(t,\cdot) \Vert_{L^{2}} = \exp(\sigma^{(j)} t)$.

The lemma is then proven if we pick $j$ large enough so that $\sigma^{(j)} \geq j^{2} C_{*}$ for some positive constant $C_{*}$ (independent of $j$). This is guaranteed by the lower bound \eqref{eq:SN:5.27} for $\sigma^{(j)}$, if we pick $j \geq j_{0}$, where $j_{0}$ is a large enough fixed integer such that
\begin{align} 
C_{*} = \frac{a \mu m (j_{0}^{4} + j_{0}^{2})}{2^{3} \Omega^{2} m^{2} (j_{0}^{4} + j_{0}^{2} + 4 m^{2}) + 2\mu^{2} j_{0}^{4}} - \frac{\kappa}{j_{0}^{2}} (j_{0}^{4}+j_{0}^{2} + 4m^{2})^{\gamma} > 0. \label{eq:SN:5.29}
\end{align}
  Note that when $\gamma \in(0,1/2)$ such a $j_{0}$ always exists, independently of the size of the physical parameters.
\end{proof}

We now have all necessary ingredients to conclude the proof of Theorem~\ref{thm:SN:5.1}.
\begin{proof}[Proof of Theorem~\ref{thm:SN:5.1}]
Let $T>0$ and $K>0$ be arbitrary, and let $Y \subset L^{2}$ be a Sobolev space. Pick an integer $m\geq 1$, and an amplitude $a>0$. For these fixed $a,m$ and physical parameters $\mu,\Omega,\kappa,\gamma>0$, define $j_0$ and $C_*$ as in Lemma~\ref{lemma:SN:5.3}. 
For any $j \geq j_0$, Lemma~\ref{lemma:SN:5.3} guarantees that there exists a $C^\infty$ smooth initial condition $\theta^{(j)}(0,\xx)$ which we re-normalize to have $\Vert \theta^{(j)}(0,\cdot) \Vert_{Y} =1$, such that the associated solution $\theta^{(j)}(t,\xx)$ of the Cauchy problem \eqref{eq:linearequation}--\eqref{eq:steadystatedef} satisfies 
\begin{align} 
\Vert \theta^{(j)}(t,\cdot) \Vert_{L^2} \geq \exp( j^2 C_* t) \Vert \theta^{(j)}(0,\cdot) \Vert_{L^2} \label{eq:SN:thm:lowerbound}
\end{align}
for all $t>0$. We note that this solution to the linear equation is the unique\footnote{The proof of uniqueness of finite energy solutions to the linear equation follows from a representation of the solution as a Fourier series (see \cite{FriedlanderVicol11c}, Proposition~2.8).} solution in $L^\infty(0,T;L^2)$. We now claim that there exists a sufficiently large $j_* \geq j_0$ such that 
\begin{align} 
\Vert \theta^{(j_*)}(T/2,\cdot) \Vert_{L^2} \geq 2  K \label{eq:SN:thm:linear}
\end{align}
which would then conclude the proof of the Theorem, since $\Vert \theta^{(j_*)}(0,\cdot) \Vert_{Y} =1$. After a short calculation it is clear that \eqref{eq:SN:thm:linear} follows from \eqref{eq:SN:thm:lowerbound} if we manage to prove that 
\begin{align} 
 \exp\left( \frac{j_*^2 C_* T}{2}\right)  \Vert \theta^{(j_*)}(0,\cdot) \Vert_{L^2} \geq 2 K. \label{eq:SN:thm:suff:cond}
\end{align}
But  from the definition of $\theta^{(j_*)}$ cf.~\eqref{eq:SN:5.2} above, we see that $\Vert \theta^{(j_*)}(0,\cdot) \Vert_{L^2} \geq c_1 = \alpha_1 \geq 1/ (C j_*^2)$, for some sufficiently large constant $C$, cf.~\eqref{eq:SN:5.4} above. The fact that for every fixed $T>0$ and $C_*>0$ the sequence $\exp(j^2 C_* T/2)/ (C j^2)$ diverges as $j\rightarrow \infty$, shows that there exits some sufficiently large $j_*$ such that \eqref{eq:SN:thm:suff:cond} holds, thereby concluding the proof of the theorem.
\end{proof}

\subsection{Non-linear ill-posedness in Sobolev spaces} \label{eq:sec:non-linear:ill} 
Having established that the {\em linearized} \MG are ill-posed in $L^2$, we now turn to address the Hadamard ill-posedness of the {\em full nonlinear} \MG equations. The ill-posedness of non-linear partial differential equations may have different sources, of varying degree of gravity: finite time blow-up, non-uniqueness, or non-smoothness of the solution map, to name a few. As in \cite{FriedlanderVicol11c} for the case $\kappa=0$, here we show that the fractionally diffusive \MG equations, with $0 <\gamma < 1/2$ do not posses a Lipschitz continuous solution map.
Recall the definition of Lipschitz well-posedness for the nonlinear equation (see~\cite[Definition 1.1]{GuoNguyen10}, or \cite[Definition 2.9]{FriedlanderVicol11c}):

\begin{definition}[\bf Lipschitz local well-posedness] \label{def:well}
Let $Y \subset X \subset W^{1,4}$ be Banach spaces. The Cauchy problem for the \MG equation
\begin{align}
&\partial_t \Th + \Ub \cdot \nabla \Th + \kappa (-\Delta)^\gamma \Th = S  \label{eq:nonlinear:ill:1}\\
&\nabla \cdot \Ub =0, \  U_{j} =  M_{j} \Th \label{eq:nonlinear:ill:2}
\end{align}is locally Lipschitz $(X,Y)$ {\em well-posed}, if there exist continuous functions $T,K: [0,\infty)^{2} \rightarrow (0,\infty)$,  the time of existence and the Lipschitz constant, so that for every pair of initial data $\Th^{(1)}(0,\cdot), \Th^{(2)}(0,\cdot) \in Y$ there exist unique solutions $\Th^{(1)}, \Th^{(2)} \in L^{\infty}(0,T;X)$ of the initial value problem associated to  \eqref{eq:nonlinear:ill:1}--\eqref{eq:nonlinear:ill:2}, that additionally satisfy
\begin{align}
\Vert \Th^{(1)}(t,\cdot) -\Th^{(2)}(t,\cdot) \Vert_{X} \leq K \Vert \Th^{(1)}(0,\cdot) - \Th^{(2)}(0,\cdot) \Vert_{Y}\label{eq:def:well}
\end{align}
for every $t\in [0,T]$, where $T = T(\Vert \Th^{(1)}(0,\cdot) \Vert_{Y}, \Vert \Th^{(2)}(0,\cdot)\Vert_{Y})$ and $K  = K (\Vert \Th^{(1)}(0,\cdot)\Vert_{Y},\Vert \Th^{(2)}(0,\cdot)\Vert_{Y})$.
\end{definition}
\begin{remark}
 Clearly the time of existence $T$ and the Lipschitz constant $K$ also depend on $\Vert S \Vert_{L^\infty(0,\infty;Y)}$, and on $\kappa$, but we have omitted this dependence in Definition~\ref{def:well} since it is the same for both solutions $\Th^{(1)}$ and $\Th^{(2)}$.
\end{remark}
\begin{remark}
 If $\Th^{(2)}(t,\cdot)=0$ and $X=Y$, Definition~\ref{def:well} recovers the usual definition of local well-posedness with a continuous solution map. However, Defintion~\ref{def:well} allows the solution map to lose regularity, which is usually needed in order to obtain Lipschitz continuity of the solution map for first order quasi-linear equations. Hence, the typical pairs of spaces $(X,Y)$ that we have in mind here are $X = H^{s}$, and  $Y=H^{s+1}$, with $s>1 + 3/4$. 
 \end{remark}

The main result of this section is the following theorem.

\begin{theorem}[\bf Nonlinear ill-posedness in Sobolev spaces] \label{thm:nonlinear:ill}
The \MG equations \eqref{eq:nonlinear:ill:1}--\eqref{eq:nonlinear:ill:2}, with $\gamma \in (0,1/2)$ are locally Lipschitz $(X,Y)$ {\em ill-posed} in Sobolev spaces $Y\subset X$ embedded in $W^{1,4}(\TT^3)$, in the sense of Definition~\ref{def:well}.
\end{theorem}

 For the purpose of our ill-posedness result, we shall let $\Th^{(2)}(t,\xx)$ be the steady state $\Th(x_{3})$ introduced earlier in \eqref{eq:steadystatedef}. We consider $X$ to be a Sobolev space with high enough regularity so that $\partial_{t} \Th \in L^{\infty}(0,T;L^{2})$, which implies that $\Th$ is weakly continuous on $[0,T]$ with values in $X$, making sense of the initial value problem associated to \eqref{eq:nonlinear:ill:1}--\eqref{eq:nonlinear:ill:2}. The proof of Theorem~\ref{thm:nonlinear:ill} follows from the strong {\em linear} ill-posedness obtained in Theorem~\ref{thm:SN:5.1} and a fairly generic perturbative argument (cf.~\cite[Thorem 2]{Renardy09} or~\cite[pp.~183]{Tao06}). The proof of Theorem~\ref{thm:nonlinear:ill} follows the lines of the proof for the non diffusive problem given in \cite{FriedlanderVicol11c}. 

\begin{proof}[Proof of Theorem~\ref{thm:nonlinear:ill}]
Since the Sobolev space $X$ embeds in $H^{1}$, and $\gamma \in (0,1/2)$, the linearized operator $L \Th = - M_{3}\Th\, \partial_{3} \Th_{0} - \kappa (-\Delta)^\gamma \Th$ maps $X$ continuously into $L^{2}$, and since $X \subset W^{1,4}$, the nonlinearity $N\Th = - M_{j}\Th\, \partial_{j} \Th$ is bounded as  $\Vert N \Th \Vert_{L^{2}} \leq \Vert \nabla \Th \Vert_{L^{4}}^{2} \leq C \Vert \Th \Vert_{X}^{2}$, for some  constant $C>0$. 

Fix the steady state $\Th_{0}(x_{3}) \in Y$, as given by \eqref{eq:steadystatedef}. Also, fix a smooth function $\psi_{0} \in Y$, normalized to have $\Vert \psi_{0} \Vert_{Y} = 1$, to be chosen precisely later. The proof is by contradiction. Assume that the Cauchy problem for the  \MG equation \eqref{eq:nonlinear:ill:1}--\eqref{eq:nonlinear:ill:2} is Lipschitz locally well-posed in $(X,Y)$. Consider $\Th^{(2)}(0,\xx) = \Th_{0}(x_{3})$, so that $\Th^{(2)}(t,\xx) = \Th_{0}(x_{3})$ for any $t>0$. Also let
\begin{align*}
\Th^{\epsilon}(0,\xx) = \Th_{0}(x_{3})  + \epsilon \psi_{0}(\xx),
\end{align*} for every $0 < \epsilon < \Vert \Th_{0} \Vert_{Y}$. To simplify notation we write $\Th^{\epsilon}$ instead of $\Th^{(1,\epsilon)}$. By Definition~\ref{def:well}, for every $\epsilon$ as before there exists a positive time $T = T (\Vert \Th_{0}\Vert_{Y},\Vert \Th^{\epsilon}\Vert_{Y})$ and a positive Lipschitz constant $ K = (\Vert \Th_{0}\Vert_{Y},\Vert \Th^{\epsilon}\Vert_{Y})$ such that by \eqref{eq:def:well} and the choice of $\psi_{0}$ we have
\begin{align}
\Vert \Th^{\epsilon}(t,\cdot) - \Th_{0}(\cdot) \Vert_{X} \leq K \epsilon \label{eq:Lip:bound}
\end{align}for all $t\in [0,T]$. We note that since $\Vert \Th^{\epsilon}(0,\cdot) \Vert_{Y} \leq \Vert \Th_{0} \Vert_{Y} + \epsilon \leq 2 \Vert \Th_{0} \Vert_{Y}$, due to the continuity of $T$ and $K$ with respect to the second coordinate, we may choose $K = K(\Vert \Th_{0} \Vert_{Y}) >0$ and $T =T (\Vert \Th_{0} \Vert_{Y})>0$ independent of $\epsilon \in (0,\Vert \Th_{0} \Vert_{Y})$, such that  \eqref{eq:Lip:bound} holds on $[0,T]$.

Writing $\Th^{\epsilon}$ as an $\OO(\eps)$  perturbation of $\Theta_0$, we define 
\begin{align}
\psi^{\epsilon}(t,\xx) = \frac{\Th^{\epsilon}(t,\xx) - \Theta_0(x_3)}{\eps},
\end{align}
for all $t \in [0,T]$ and all $\epsilon$ as before.  It follows from \eqref{eq:Lip:bound} that $ \psi^\epsilon$   is uniformly bounded with respect to $\epsilon$ in $L^{\infty}(0,T;X)$. Therefore, there exists a function $\psi$, the weak-$*$ limit of $\psi^{\epsilon}$ in $L^\infty(0,T;X)$. Note that $\psi^{\epsilon}$ solves the Cauchy problem
  \begin{align}
    \partial_t \psi^\epsilon &= L \psi^\epsilon + {\epsilon} N(\psi^\epsilon) \label{eq:psieps}\\
    &\psi^{\epsilon}(0,\cdot) = \psi_{0}.
  \end{align}
  Due to the choice of $X$, we have the bound
  \begin{align}
  \Vert N \psi^{\epsilon} \Vert_{L^{2}} \leq C \Vert \psi^{\epsilon} \Vert_{X}^{2} \leq C K^{2} \label{eq:Nbound},
  \end{align}and from \eqref{eq:psieps} we obtain that $\partial_{t} \psi^{\epsilon}$ is uniformly bounded with respect to $\epsilon$ in $L^{\infty}(0,T;L^{2})$. Therefore the convergence $\psi^{\epsilon} \rightarrow \psi$ is strong when measured in $L^{2}$. Sending $\epsilon$ to $0$ in \eqref{eq:psieps}, and using \eqref{eq:Nbound}, it follows that
  \begin{align*}
    \partial_t \psi &= L \psi\\
    \psi(0,\cdot)&= \psi_{0}
  \end{align*}
  holds in $L^\infty(0,T;L^2)$, and this solution is unique since the problem is now linear. In addition, the solution $\psi$ inherits from \eqref{eq:Lip:bound} the upper bound
  \begin{align}
  \Vert \psi(t,\cdot) \Vert_{L^{2}} \leq K \label{eq:Lip:2}
  \end{align} for all $t\in[0,T]$. But this is a contradiction with Theorem~\ref{thm:SN:5.1}. Due to the existence of eigenfunctions for the linearized operator with arbitrarily large eigenvalues, one may choose $\psi_{0}$ (as in Lemma~\ref{lemma:SN:5.3}) to yield a large enough eigenvalue so that in time $T/2$ the solution grows to have $L^{2}$ norm larger than $2K$ (similarly to the proof of Theorem~\ref{thm:SN:5.1}), therefore contradicting \eqref{eq:Lip:2}.
\end{proof}

\section{The regime $\gamma = 1/2$: a dichotomy in terms of the size of the data} \label{sec:=1/2}

If the initial data is small with respect to $\kappa$, then one may use energy estimates to show that there exists a unique global smooth solution evolving from this data (cf.~Section~\ref{sec:=1/2_small_data} below). However, the proof does not apply for the case of large data, not even to prove the local existence of solutions. In the case of large data, we may construct a steady state such that solutions evolving from initial data which is arbitrarily close to this steady state diverge from it at an arbitrarily large rate, for positive time (cf.~Section~\ref{sec:=1/2_large_data} below), i.e. the equations are Hadamard ill-posed.

\subsection{Ill-posedness for $\gamma=1/2$ and large data} \label{sec:=1/2_large_data} 
In Section~\ref{sec:<1/2} above we have shown that for $\gamma<1/2$ the \MG equations are Hadamard ill-posed in Sobolev spaces, in the sense that there is no Lipschitz continuous solution map (see Theorem~\ref{thm:nonlinear:ill}). The main ingredient in the proof of ill-posedness for $\gamma \in (0,1/2)$ was to show that the linearized \MG operator
\begin{align}
L \theta = - M_3 \theta \partial_3 \Th_0 - \kappa (-\Delta)^\gamma \theta
\end{align}
where
\begin{align}
\Th_0 = a \sin(m x_3)
\end{align}
is a steady state associated to the source term $S =\kappa a m^{2\gamma} \sin(m x_3)$, has a sequence of eigenvalues $\sigma^{(j)}$ which diverge to infinity as $j \rightarrow \infty$ (see Lemmas~\ref{lemma:SN:5.2} and \ref{lemma:SN:5.3}). 

We emphasize that it is {\em only} in the proof of Lemma~\ref{lemma:SN:5.3} where $\gamma < 1/2$, rather than $\gamma \leq 1/2$ is used. Indeed, in order to prove that given {\em any} real $a>0$ and any integer $m \geq 1$ there exists some sufficiently large integer $j_0$ such that the constant $C_*$ of \eqref{eq:SN:5.29} is strictly positive, $\gamma < 1/2$ is both necessary and sufficient. In the case $\gamma=1/2$, we can prove the following alternative to Lemma~\ref{lemma:SN:5.3}, which states that only if $a$ is sufficiently large with respect to $\kappa$ (in terms of $\mu$ and $\Omega$), then one obtains a sequence of eigenvalues which diverge to $\infty$. Namely:

\begin{lemma} \label{lemma:SN:5.4}
Let $\gamma = 1/2$. Fix $a,\mu,\Omega$, and $m$. For values of $a$ such that
\begin{align} 
\kappa < \frac{a \mu m}{2^5 \Omega^2 m^2 + 2 \mu^2} \label{eq:SN:cond:a}
\end{align}
and all integers $j \geq m$, the statement of Lemma~\ref{lemma:SN:5.3} holds with the constant $C_*$ given by
\begin{align} 
C_*  = \frac{a \mu m}{2^5 \Omega^2 m^2 + 2 \mu^2} - \kappa \label{eq:SN:5.30}.
\end{align}
\end{lemma}
\begin{proof}[Proof of Lemma~\ref{lemma:SN:5.4}]
 The proof follows from the lower bound \eqref{eq:SN:5.27} on the eigenvalue $\sigma^{(j)}$, similarly to the proof of Lemma~\ref{lemma:SN:5.3}. The role of condition \eqref{eq:SN:cond:a} is to guarantee that there exists a large enough $j_0$ such that 
\begin{align*}
 \frac{a \mu m (j_{0}^{4} + j_{0}^{2})}{2^{3} \Omega^{2} m^{2} (j_{0}^{4} + j_{0}^{2} + 4 m^{2}) + 2\mu^{2} j_{0}^{4}} - \frac{\kappa}{j_{0}^{2}} (j_{0}^{4}+j_{0}^{2} + 4m^{2})^{\gamma} > 0.
\end{align*}
To avoid repetition we omit further details.
\end{proof}
For those values of $a$ for which \eqref{eq:SN:cond:a} holds, Lemma~\ref{lemma:SN:5.4} shows that the operator $L$ has unbounded spectrum in the unstable region, and hence one may prove with virtually {\em no modifications} from the $\gamma \in (0,1/2)$ case, that the equations are ill-posed. More precisely we have
\begin{theorem}[\bf Ill-posedness for large data] \label{thm:ill:1/2}
 Let $Y \subset X \subset W^{1,4}(\TT^3)$ be Sobolev spaces, and let $\gamma=1/2$. Given $\kappa, \mu, \Omega >0$, fix an integer $m\geq 1$. Let $a>0$ be sufficiently large such that \eqref{eq:SN:cond:a} holds and let $\Th_0 ^{(1)}= a \sin(m x_3) \in Y$ be a steady state of \eqref{eq:1}. Then, given any $T>0$ and any $K>0$, there exists an initial condition $\Th_0^{(2)} \in Y$ and a corresponding solution $\Th^{(2)} \in L^\infty(0,T;X)$ of the Cauchy problem \eqref{eq:1}--\eqref{eq:3}, such that 
 \begin{align}
\sup_{t\in [0,T]} \Vert  \Th^{(2)}(t,\cdot) - \Th^{(1)}(t,\cdot) \Vert_{X} \geq 2 K \Vert  \Th^{(2)}_0 - \Th^{(1)}_0 \Vert_{Y}
\end{align}
where $ \Th^{(1)}(t,\cdot) = \Th^{(1)}_0$. That is, the \MG equations are locally Lipschitz (X,Y) ill-posed when the data is large.
\end{theorem}
The proof of Theorem~\ref{thm:ill:1/2} is the same as the proof of Theorem~\ref{thm:nonlinear:ill} above. The only difference in the case $\gamma=1/2$ is to use Lemma \ref{lemma:SN:5.4} to show that the linear equations have unbounded unstable spectrum. We omit  details.

\subsection{Well-posedness for $\gamma=1/2$ and small data} \label{sec:=1/2_small_data}

As shown in the previous section, we can exhibit initial data, for which the system (\ref{eq:1})--(\ref{eq:2}) is ill-posed in the above described sense. Note also, that the proof of Theorem \ref{thm:121:local}, fails for the value $\gamma = 1/2$. Thus, $\gamma=1/2$ indeed is the limit of the local well-posedness theory. Nonetheless, we still can prove that the considered system is globally well-posed for small data.

\begin{theorem}[{\bf Global existence for small data}] \label{thm:gamma=1/2:small}
	Let $s>5/2$ and assume that the initial data $\Th_0 \in H^s(\TT^3)$ and $S \in L^2(0,\infty;H^{s-\gamma}(\TT^3))$ have zero-mean on $\TT^{3}$. There exists a sufficiently small constant $\varepsilon>0$ depending on $\kappa$, such that if $\Vert \Th_0\Vert_{L^2}^\alpha\Vert \Th_{0} \Vert_{H^s}^{1-\alpha} + \Vert S \Vert_{L^\infty(0,\infty;H^{s-\gamma})} \leq \varepsilon$, where $\alpha=1-(5/2)/s$, then the unique smooth solution 
	\begin{align} 
		\Th \in L^\infty(0,T;H^s(\TT^3))
	\end{align}	
	 of the Cauchy problem \eqref{eq:1}--\eqref{eq:M:3} is global in time.
\end{theorem}
\begin{proof}
	We proceed as in the proof of Theorem \ref{thm:121:global} and obtain the energy estimate
	\begin{align}\label{eq:3.1}
		&\frac{1}{2}\frac{d}{dt}\|\Lambda^s\Th\|^2_{L^2} + \kappa\|\Lambda^{s+\frac{1}{2}}\Th\|^2_{L^2} \leq \|\Lambda^{s-\frac{1}{2}}S\|_{L^2}\|\Lambda^{s+\frac{1}{2}}\Th\|_{L^2} + \left| \int_{\RR^3} \Lambda^{s-\frac{1}{2}}(\Ub\cdot \nabla \Th)\Lambda^{s+\frac{1}{2}}\Th \right|.
	\end{align}
The second term on the right side is estimated using the product estimate in Proposition \ref{prop:calculus}. Thus we obtain, similarly to (\ref{eq:2.8})
\begin{align}
	&\|\Lambda^{s-\frac{1}{2}}(\Ub \cdot \nabla \Th)\|_{L^2} \leq C(\|\nabla \Th\|_{L^\infty} + \|\Ub\|_{L^\infty})\|\Lambda^{s+\frac{1}{2}}\Th\|_{L^2}.
\end{align}
Since $s > 5/2$,  we get
\begin{align}\label{eq:3.3}
	& \|\Ub\|_{L^\infty} + \|\nabla \Th\|_{L^\infty} \leq C\|\Th\|_{L^2}^\alpha\|\Lambda^s \Th\|_{L^2}^{1-\alpha},
\end{align}
where $\alpha = 1-(n/2+1)/s$. 
Combining estimates (\ref{eq:3.1})-(\ref{eq:3.3}) gives
\begin{align}
	&\frac{1}{2}\frac{d}{dt}\|\Lambda^s\Th\|^2_{L^2} + \kappa\|\Lambda^{s+\frac{1}{2}}\Th\|^2_{L^2} \leq \frac{1}{2\kappa}\|\Lambda^{s-\frac{1}{2}}S\|^2_{L^2} + \frac{\kappa}{2}\|\Lambda^{s+\frac{1}{2}}\Th\|^2_{L^2} + C\|\Th\|_{L^2}^\alpha\|\Lambda^s \Th\|_{L^2}^{1-\alpha}\|\Lambda^{s+\frac{1}{2}} \Th\|_{L^2}^2,
\end{align}
which in turn leads to
\begin{align}\label{eq:3.5}
	&\frac{1}{2}\frac{d}{dt}\|\Lambda^s\Th\|^2_{L^2} + \frac{\kappa}{2}\|\Lambda^{s+\frac{1}{2}}\Th\|^2_{L^2} \leq \frac{1}{2\kappa}\|\Lambda^{s-\frac{1}{2}}S\|^2_{L^2} +  C\|\Th\|_{L^2}^\alpha\|\Lambda^s \Th\|_{L^2}^{1-\alpha}\|\Lambda^{s+\frac{1}{2}} \Th\|_{L^2}^2.
\end{align}
We obtain the desired result as in the proof of Theorem \ref{thm:121:global}.
\end{proof}
\begin{remark}
	Note that conditions in Theorem \ref{thm:121:global} are consistent with the above theorem if we set $\gamma=1/2$. 
\end{remark}

\begin{remark}
 We note that the ill-posedness result in the case $\gamma=1/2$ requires the constant defined in \eqref{eq:SN:5.30} to be positive. This does not hold when the value of $a$ is sufficiently small (depending on $\kappa, \mu$, and $\mu$) no matter what value of $m\geq 1$ is picked. Recalling that $a$ measures the magnitude of the initial data associated with $\Vert \Th_0 \Vert_{L^2}$, we observe that the well-posedness result  when $\gamma = 1/2$ is only proven in the case of small data and small force (see Theorem~\ref{thm:gamma=1/2:small}), and hence for $a$ sufficiently small. Therefore our large data ill-posedness result is consistent with the small-data well-posedness result when $\gamma = 1/2$.
\end{remark}

\section{Improvement in regularity for ``well-prepared'' data and source} \label{sec:well-prepared}

In this section we explore the following observation:  if the frequency support of $\Th$ lies on a suitable plane, then the operator $\MM$ is ``mild'' when it acts on $\Th$, i.e. it behaves like an order $0$ operator, and hence the corresponding velocity $\Ub$ is as smooth as $\Th$. This enables us to obtain improved well-posedness results over the generic setting when no conditions on the Fourier spectrum of the initial data (and source term) are imposed. For instance the local existence an uniqueness of smooth solutions holds even if $0<\gamma<1/2$, a setting in which we know that for generic initial data the problem is ill-posed.

To be more precise let $q \in {\mathbb Q}$ be a rational number\footnote{Since we are in the periodic setting when the frequency is supported on $\ZZ^{3}$.} with $q \neq 0$. We define the frequency plane $\PP_{q}$ to be the set
\begin{align}
\PP_{q} = \{ \kk = (k_{1},k_{2},k_{3}) \in \ZZ^{3} \colon k_{2} = q k_{1} \}. \label{eq:Pq:def}
\end{align}
We shall need the following straightforward observation:
\begin{proposition}
\label{prop:support}
Assume $f,g$ are smooth $\TT^{3}$ periodic functions, with frequency support $\supp(\hat f), \supp(\hat g) \subset \PP_{q}$ for some $q \in {\mathbb Q}\setminus\{0\}$. Then we have $\supp(\hat{f\pm g}), \supp( \hat{fg} ), \supp (\hat {M_{j} f}) \subset \PP_{q}$, for all $j \in \{1,2,3\}$.
\end{proposition}
\begin{proof}[Proof of Propostion~\ref{prop:support}]
Clearly
$\supp(\hat{f + g}) = \supp (\hat f + \hat g) \subset \supp (\hat f) \cup \supp (\hat g) \subset  \PP_{q}$.
To see that that the frequency support of the function $M_{j} f$ lies on  $\PP_{q}$, we just note that $\supp(\hat{M_{j} g}) = \supp( \hat{M_{j}} \hat{f})  \subset \supp(\hat{M_{j}}) \cap \supp(\hat f) \subset \PP_{q}$. To analyze the frequency support of $\hat{fg} = \hat f \ast \hat g$, note that
$\supp( \hat f \ast \hat g) \subset \supp(\hat f ) + \supp(\hat g) \subset \PP_{q}$, since $\PP_{q}$ is closed under addition of vectors, concluding the proof.
\end{proof}

\begin{lemma}\label{lemma:M:order:0}
If $f$ is a smooth $\TT^{3}$ periodic function with $\hat f(k_{1},k_{2},0) = 0$ for all $k_{1},k_{2} \in \ZZ$, and such that  $\supp(\hat f) \subset \PP_{q}$ for some $q \in {\mathbb Q}\setminus\{0\}$, then there exists a universal constant $C_{\ast} = C_{\ast} (q,\Omega,\beta,\eta)>0$ such that
\begin{align}
\left| (\hat{M_{j} f})(\kk)\right| = \left| \hat{M_{j}}(\kk) \hat{f}(\kk)\right| \leq C_{\ast}|\hat{f}(\kk)| \label{eq:M:order:0}
\end{align}
for all $\kk \in \ZZ^{3}$, and for all $j \in \{1,2,3\}$. Additionally, the constant $C_*$ blows-up as $q \rightarrow 0$.
\end{lemma}
\begin{proof}[Proof of Lemma~\ref{lemma:M:order:0}] It is clear that \eqref{eq:M:order:0} has to be proven only for $\kk$ such that $k_{3} \neq 0$, since otherwise we have that $\hat{f}(\kk) = 0$ and the statement holds trivially. We now consider each of the cases $j \in \{1,2,3\}$. Without loss of generality take $q>0$.

For $j=1$, a short algebraic computations gives
\begin{align}
\left |\hat{M}_{1} (\kk) \right| &\leq \left( 2\Omega (2 q + q^{3})  + q^{2} \beta^{2}/\eta \right) \frac{k_{1}^{3} k_{3} + k_{1} k_{3}^{3}}{2\Omega^{2} k_{3}^{4}+ q^{4} \beta^{4}/\eta^{2} k_{1}^{4}} \label{eq:M1:estimate:order:0}
\end{align}
from which it follows that $|\hat{M}_{1}(\kk)| \leq C$, for a suitable constant $C$, by using the inequality $k_{1}^{3} k_{3} + k_{1} k_{3}^{3} \leq k_{1}^{4}  + k_{3}^{4}$. Similarly to \eqref{eq:M1:estimate:order:0} it follows that for $\kk \in \PP_{q}$ we have
\begin{align}
\left| \hat{M}_{2} (\kk)\right| \leq \left( 2\Omega (2+q^{2})+ q^{3} \beta^{2}/\eta \right) \frac{k_{1}^{3} k_{3} + k_{1} k_{3}^{3}}{2\Omega^{2} k_{3}^{4}+ q^{4} \beta^{4}/\eta^{2} k_{1}^{4}}
\end{align}
and
\begin{align}
\left| \hat{M}_{3} (\kk)\right| \leq  \left( q^{2} (1+q^{2}) \beta^{2}/\eta\right) \frac{k_{1}^{4} }{2\Omega^{2} k_{3}^{4}+ q^{4} \beta^{4}/\eta^{2} k_{1}^{4}}
\end{align}
which concludes the proof of the lemma upon using $k_{1}^{3} k_{3} + k_{1} k_{3}^{3} \leq k_{1}^{4}  + k_{3}^{4}$.
\end{proof}

\subsection{Local existence and uniqueness for $0 < \gamma < 1$}\label{sec:well-prepared:local}
The main result of this section is:

\begin{theorem}[{\bf Local existence}] \label{thm:well-prepared:local}
Let $\gamma \in (0,1)$, and fix $s>5/2-\gamma$. Assume that $\Th_0 \in H^s(\TT^3)$ and $S \in L^\infty(0,\infty;H^{s-\gamma}(\TT^3))$ have zero-mean on $\TT^{3}$ and are such that
\begin{align}
\supp(\hat\Th_{0}) \subset \PP_{q} \mbox{ and } \supp(\hat{S}(t,\cdot)) \subset  \PP_{q},
\end{align} for some fixed $q \in {\mathbb Q}\setminus\{0\}$ and all $t\geq 0$. Then there exists a $T>0$ and a unique smooth solution 
\begin{align} \Th \in L^\infty(0,T;H^s(\TT^3)) \cap L^2(0,T;H^{s+\gamma}(\TT^3))
\end{align} of the Cauchy problem  \eqref{eq:FractionalMG:1}--\eqref{eq:FractionalMG:2}, such that
\begin{align}
  \supp(\hat\Th (t,\cdot)) \subset \PP_{q} 
\end{align}
for all $t\in[0,T)$.
\end{theorem}

\begin{proof}[Proof of Theorem~\ref{thm:well-prepared:local}]
In order to construct the local in time solution $\Th$, with frequency support on $\PP_q$, consider the sequence of approximations $\{ \Th_n \}_{n\geq 1}$ given by the solutions of
\begin{align}
  &\partial_t \Th_1 + (-\Delta)^\gamma \Th_1 = S  \label{eq:Th:1:1}\\
  &\Th_1(0,\cdot) = \Th_0 \label{eq:Th:1:2}
\end{align}
and respectively
\begin{align}
  &\partial_t \Th_n + \Ub_{n-1} \cdot \nabla \Th_n + (-\Delta)^\gamma \Th_n  = S \label{eq:Th:n:1}\\
  & \Ub_{n-1} = \MM \Th_{n-1} \label{eq:Th:n:2}\\
  &\Th_n(0,\cdot)= \Th_0 \label{eq:Th:n:3}
\end{align}
for all $n \geq 2$.
One may solve \eqref{eq:Th:1:1}--\eqref{eq:Th:1:2} explicitly using Duhamel's formula, and noting that the operator $\exp( (-\Delta)^\gamma t)$ is a Fourier multiplier with non-zero symbol given by $\exp(|\kk|^\gamma t)$, it follows from Proposition~\ref{prop:support} that $\supp(\hat{\Th_1}(t,\cdot)) \subset \PP_q$ for all $t\geq 0$. In addition, we have the bound
\begin{align}
\Vert \Lambda^s \Th_1 \Vert_{L^\infty(0,T;L^2)}^2 + \Vert \Lambda^{s+\gamma} \Th_1 \Vert_{L^2(0,T;L^2)}^2 \leq \Vert \Lambda^s \Th_0 \Vert_{L^2}^2 + T \Vert S \Vert_{L^\infty(0,T;H^{s-\gamma})}^2 \label{eq:Th:1:bound}
\end{align} for all $T>0$.
On the other hand, in order to solve \eqref{eq:Th:n:1}--\eqref{eq:Th:n:3} we appeal to Theorem~\ref{thm:app:linear}. Indeed, by the inductive assumption we have that $\Th_{n-1} \in L^\infty(0,T;H^s)$  and also $\supp(\hat{\Th_{n-1}}(T,\cdot) ) \subset \PP_q$ for all $T>0$. Hence, by applying Lemma~\ref{lemma:M:order:0} we have $\Ub_{n-1} \in L^\infty(0,T;H^s)$ and by Proposition~\ref{prop:support} we have $\supp(\hat{\Ub_{n-1}}(t,\cdot))$ for all $t>0$. Therefore all the conditions of Theorem~\ref{thm:app:linear} are satisfied, by letting $v = \Ub_{n-1}$, and there hence exists a unique solution $\Th_n \in L^\infty(0,T;H^s) \cap L^2(0,T;H^{s+\gamma})$ of \eqref{eq:Th:n:1}--\eqref{eq:Th:n:3}, such that $\supp(\hat{\Th_n}(t,\cdot)) \subset \PP_q$ for all $t\in [0,T)$.

To prove that the sequence $\{\Th_n\}$ converges, we first prove that it is bounded. Fix a time $T$ to be chosen later such that $T < \Vert \Lambda^s \Th_0 \Vert_{L^2}^2 / \Vert S \Vert_{L^\infty(0,T;H^{s-\gamma})}^2$. It hence follows from \eqref{eq:Th:1:bound} that
\begin{align}
  \Vert \Lambda^s \Th_j \Vert_{L^\infty(0,T;L^2)}^2 + \Vert \Lambda^{s+\gamma} \Th_j \Vert_{L^2(0,T;L^2)}^2 \leq 2 \Vert \Lambda^s \Th_0 \Vert_{L^2}^2 \label{eq:Th:inductive:bound}
\end{align}
holds when $j=1$. We assume inductively that \eqref{eq:Th:inductive:bound} holds for all $1 \leq j \leq n-1$, and proceed to prove that it holds for $j=n$. Since $\Ub_{n-1}$ is obtained from $\Th_{n-1}$ by a bounded Fourier multiplier (cf.~Lemma~\ref{lemma:M:order:0}), there exists a positive constant $C_q$, depending on $q$ (and on the physical parameters $\Omega, \beta, \eta$), such that 
\begin{align}
  \Vert \Lambda^s \Ub_{n-1}(t,\cdot) \Vert_{L^2}^2 \leq C_q  \Vert \Lambda^s \Th_{n-1}(t,\cdot) \Vert_{L^2}^2 \label{eq:U:n-1:bound}
\end{align}
for all $t\geq 0$.
Applying $\Lambda^s$ to \eqref{eq:Th:n:1} and taking an $L^2$ inner product with $\Lambda^s \Th_n$ we hence obtain
\begin{align}
& \frac{1}{2} \frac{d}{dt} \Vert \Lambda^{s} \Th_n \Vert_{L^{2}}^{2} + \Vert \Lambda^{s+\gamma} \Th_n \Vert_{L^{2}}^{2} \notag\\
& \qquad \leq \Vert \Lambda^{s} \Th_n \Vert_{L^{2}} \Vert [ \Ub_{n-1} \cdot \nabla, \Lambda^{s}] \Th_n  \Vert_{L^{2}} + \Vert \Lambda^{s+\gamma} \Th_n \Vert_{L^{2}} \Vert \Lambda^{s-\gamma} S \Vert_{L^{2}}  \notag\\
& \qquad \leq C \Vert \Lambda^{s}  \Th_n \Vert_{L^{2}} \left( \Vert \nabla \Ub_{n-1} \Vert_{L^{3/\gamma}} \Vert \Lambda^{s} \Th_n \Vert_{L^{6/(6-\gamma)}} + \Vert \Lambda^{s} \Ub_{n-1} \Vert_{L^{2}}\Vert  \nabla \Th_n \Vert_{L^{\infty}}\right) + \Vert \Lambda^{s+\gamma} \Th_n \Vert_{L^{2}} \Vert \Lambda^{s-\gamma} S \Vert_{L^{2}}  \notag\\
&\qquad \leq C \Vert \Lambda^{s} \Th_n \Vert_{L^{2}} \Vert \Lambda^{s} \Ub_{n-1} \Vert_{L^{2}} \Vert \Lambda^{s+\gamma} \Th_n \Vert_{L^{2}} + \Vert \Lambda^{s+\gamma} \Th_n \Vert_{L^{2}} \Vert \Lambda^{s-\gamma} S \Vert_{L^{2}}  \notag\\
&\qquad \leq \frac{1}{2} \Vert \Lambda^{s+\gamma} \Th_n \Vert_{L^{2}}^{2} + C_{s,\gamma} \Vert \Lambda^{s} \Ub_{n-1} \Vert_{L^{2}}^{2} \Vert \Lambda^{s} \Th_n \Vert_{L^{2}}^{2} + \Vert \Lambda^{s-\gamma} S \Vert_{L^{2}}^{2} \label{eq:Th:n:bound:1}
\end{align}
where $C_{s,\gamma}> 0$ is a constant. Above we have used the fact that $\nabla \cdot \Ub_{n-1} = 0$ in order to write the commutator, and also the product and commutator estimates of Proposition~\ref{prop:calculus}. Using the bound \eqref{eq:U:n-1:bound}, and the inductive assumption \eqref{eq:Th:inductive:bound}, it  follows from \eqref{eq:Th:n:bound:1} that
\begin{align}
 \frac{d}{dt} \Vert \Lambda^{s} \Th_n \Vert_{L^{2}}^{2} + \Vert \Lambda^{s+\gamma} \Th_n \Vert_{L^{2}}^{2} \leq 4 C_{s,\gamma} C_q \Vert \Lambda^s \Th_0 \Vert_{L^2}^2 \Vert \Lambda^s \Th_n \Vert_{L^2}^2 + 2 \Vert \Lambda^{s-\gamma} S \Vert_{L^2}^2
\end{align}
and hence
\begin{align}
  &\Vert \Lambda^{s} \Th_n(t,\cdot) \Vert_{L^{2}}^{2} + \int_0^t  \Vert \Lambda^{s+\gamma} \Th_n(\tau,\cdot) \Vert_{L^2}^2 d\tau\notag\\
   & \qquad \leq \Vert \Lambda^s \Th_0 \Vert_{L^2}^2 e^{4 C_{s,\gamma} C_q \Vert \Lambda^s \Th_0 \Vert_{L^2}^2 t} + \int_0^t e^{4 C_{s,\gamma} C_q \Vert \Lambda^s \Th_0 \Vert_{L^2}^2 (t-\tau)} \Vert \Lambda^{s-\gamma} S(\tau,\cdot) \Vert_{L^2}^2 d\tau \notag\\
   & \qquad \leq e^{4 C_{s,\gamma} C_q \Vert \Lambda^s \Th_0 \Vert_{L^2}^2 t}\left(\Vert \Lambda^s \Th_0 \Vert_{L^2}^2  + t \Vert S \Vert_{L^\infty(0,\infty;H^{s-\gamma})}^2 \right) \label{eq:Th:n:bound:2}
\end{align}
for all $t\in (0,T)$. Therefore, letting
\begin{align}
  T \leq \min \left\{ \frac{\Vert \Lambda^s \Th_0 \Vert_{L^2}^2}{2 \Vert S \Vert_{L^\infty(0,\infty;H^{s-\gamma})}^2}, \frac{\ln 4/3}{4 C_{s,\gamma} C_q \Vert \Lambda^s \Th_0 \Vert_{L^2}^2} \right\} \label{eq:T:choice:1}
\end{align}
we obtain that from \eqref{eq:Th:n:bound:2} that \eqref{eq:Th:inductive:bound} holds for $j=n$, and so by induction it holds for all $j\geq 1$. This shows that the sequence $\Th_n$ is uniformly bounded in $L^\infty(0,T;H^s) \cap L^2(0,T;H^{s+\gamma})$.


Moreover, we may show that the sequence $\{ \Th_n \}$ is Cauchy in $L^\infty(0,T;H^{s-1+\gamma})$. To see this, consider the difference of two iterates $\tilde{\Th}_n := \Th_n - \Th_{n-1}$. It follows from \eqref{eq:Th:n:1} that $\tilde{\Th}_n$ is a solution of
\begin{align}
  &\partial_t \tilde{\Th}_n + (-\Delta)^\gamma \tilde{\Th}_n + \Ub_{n-1} \cdot \nabla \tilde{\Th}_n + \tilde{\Ub}_{n-1} \cdot \nabla \Th_{n-1} = 0 \label{eq:Th:n:diff}\\
  &\tilde{\Th}_n(0,\cdot) = 0
\end{align}
for all $n\geq 2$, where we also denoted $\tilde{\Ub}_{n} = \MM \tilde{\Th}_n$. Applying $\Lambda^{s-1+\gamma}$ to \eqref{eq:Th:n:diff}, taking an $L^2$ inner product with $\Lambda^{s-1+\gamma} \tilde{\Th}_n$, using $\nabla \cdot \Ub_{n-1} =0$, and the calculus inequalities of Proposition~\ref{prop:calculus}, we obtain
\begin{align}
  &\frac 12 \frac{d}{dt} \Vert \Lambda^{s-1+\gamma} \tilde{\Th}_n \Vert_{L^2}^2 + \Vert \Lambda^{s-1+ 2\gamma} \tilde{\Th}_n \Vert_{L^2}^2 \notag\\
  & \qquad \leq \Vert \Lambda^{s-1 + \gamma} \tilde{\Th}_n \Vert_{L^2}  \Vert [ \Ub_{n-1} \cdot \nabla, \Lambda^{s-1+\gamma}] \tilde{\Th}_n \Vert_{L^2} + \Vert \Lambda^{s-1+\gamma} \tilde{\Th}_n \Vert_{L^2}  \Vert \Lambda^{s-1+\gamma} ( \tilde{\Ub}_{n-1} \cdot \nabla \Th_n) \Vert_{L^2}\notag\\
  & \qquad \leq C_{s} \Vert \Lambda^{s-1+\gamma} \tilde{\Th}_n \Vert_{L^2} \left( \Vert \nabla \Ub_{n-1} \Vert_{L^\infty} \Vert \Lambda^{s-1+\gamma} \tilde{\Th}_n \Vert_{L^2} + \Vert \Lambda^{s-1+\gamma} \Ub_{n-1} \Vert_{L^6} \Vert \nabla \tilde{\Th}_n \Vert_{L^{3}}\right) \notag\\
  & \qquad \qquad + C_{s} \Vert \Lambda^{s-1+\gamma} \tilde{\Th}_n \Vert_{L^2} \left( \Vert \tilde{\Ub}_{n-1} \Vert_{L^\infty} \Vert \Lambda^{s +\gamma} \Th_n \Vert_{L^2} + \Vert \Lambda^{s-1 +\gamma} \tilde{\Ub}_{n-1} \Vert_{L^2} \Vert \nabla \Th_n \Vert_{L^\infty} \right)\notag\\
  & \qquad \leq C_{s,q} \Vert \Lambda^{s-1 +\gamma} \tilde{\Th}_n \Vert_{L^2} \left( \Vert \Lambda^{s+\gamma}\Th_{n-1} \Vert_{L^{2}}\Vert \Lambda^{s-1 +\gamma} \tilde{\Th}_n \Vert_{L^2}  + \Vert \Lambda^{s-1+\gamma} \tilde{\Th}_{n-1} \Vert_{L^{2}} \Vert \Lambda^{s +\gamma} \Th_{n} \Vert_{L^{2}}  \right). \label{eq:Th:Lip:1}
\end{align}
In the last inequality above we have used the assumption $s> 5/2-\gamma$, and the estimate \eqref{eq:U:n-1:bound}. It hence follows from \eqref{eq:Th:Lip:1}, \eqref{eq:Th:inductive:bound}, and the Cauchy-Schwartz inequality that
\begin{align} 
\sup_{t \in [0,T]} \Vert \Lambda^{s-1+\gamma} \tilde{\Th}_n (t,\cdot) \Vert_{L^2} 
&\leq \sup_{t\in[0,T]} \Vert \Lambda^{s-1+\gamma} \tilde{\Th}_{n-1} (t,\cdot) \Vert_{L^2} 
\int_{0}^{T} \Vert \Lambda^{s+\gamma} \Th_{n}(t,\cdot) \Vert_{L^{2}} e^{\left( \int_{t}^{T} \Vert \Lambda^{s+\gamma} \Th_{n-1}(\tau,\cdot) \Vert_{L^{2}} d\tau \right) } dt \notag\\
& \leq \sup_{t\in[0,T]} \Vert \Lambda^{s-1+\gamma} \tilde{\Th}_{n-1} (t,\cdot) \Vert_{L^2} 
\sqrt{2T} C_{s,q} \Vert \Lambda^{s} \Th_{0}\Vert_{L^{2}} e^{\sqrt{2T} C_{s,q}\Vert \Lambda^{s} \Th_{0}\Vert_{L^{2}}} dt.\label{eq:Th:Lip:2}
\end{align} 
Recalling \eqref{eq:T:choice:1}, if we let $T$ be such that
\begin{align}
T =\min \left\{ \frac{\Vert \Lambda^s \Th_0 \Vert_{L^2}^2}{2 \Vert S \Vert_{L^\infty(0,\infty;H^{s-\gamma})}^2}, \frac{\ln 4/3}{4 C_{s,\gamma} C_q \Vert \Lambda^s \Th_0 \Vert_{L^2}^2}, \frac{C_{s,q}^{2}}{8 \exp(\sqrt{2 \ln 4/3} C_{s,q} / \sqrt{C_{s,\gamma} C_{q}})}  \right\} 
\end{align}
we obtain from \eqref{eq:Th:Lip:1} that
\begin{align}
\sup_{t \in [0,T]} \Vert \Lambda^{s-1+\gamma} \tilde{\Th}_n (t,\cdot) \Vert_{L^2}  &\leq \frac{1}{2} \sup_{t\in[0,T]} \Vert \Lambda^{s-1+\gamma} \tilde{\Th}_{n-1} (t,\cdot) \Vert_{L^2}.
\end{align}
Thus $\Th_{n}$ is Cauchy in $L^{\infty}(0,T;H^{s-1+\gamma})$, and hence $\Th_{n}$ converges strongly to $\Th$ in $L^{\infty}(0,T;H^{s-1+\gamma})$. Noting that $ s-1+\gamma > 3/2$, this shows that the strong convergence occurs in a H\"older space relative to $x$, which is sufficient to prove that the limiting function $\Th \in L^{\infty}(0,T;H^{s}) \cap L^{2}(0,T;H^{s+\gamma}) \cap Lip(0,T;H^{\min(s-2\gamma,s-1)}) $ is a solution
of the initial value problem \eqref{eq:1}--\eqref{eq:2}. 

To conclude the proof of the theorem, we note that if $\Th^{(1)}$ and $\Th^{(2)}$ are two solutions of \eqref{eq:1}--\eqref{eq:2}, then $\Th = \Th^{(1)} - \Th^{(2)}$ solves 
\begin{align} 
\partial_{t} \Th + (-\Delta)^{\gamma } \Th + \Ub^{(1)} \cdot \nabla \Th + \Ub \cdot \nabla \Th^{(2)} = 0 \label{eq:TH:diff}
\end{align}
with initial condition $\Th(0,\cdot) = 0$. An $L^{2}$ estimate on \eqref{eq:TH:diff} shows that $\Th(t,\cdot)= 0 $ for all $t\in[0,t)$, since $\Th^{(2)} \in L^{\infty}(0,T;H^{s})$ and since if $\Th^{(1)}$ and $\Th^{(2)}$ have frequency support on $\PP_{q}$, so does $\Th$.
\end{proof}

\begin{corollary}[{\bf Local existence for well-prepared data}] \label{cor:well-prepared}
Let $j_{1}, j_{2} \in \ZZ\setminus\{0\}$, and fix $s,\gamma$ as in the statement of Theorem~\ref{thm:well-prepared:local}. Assume that $\Th_{0} \in H^{s}(\TT^{3})$ is such that $\hat{\Th}_{0}(k_{1},k_{2},k_{3}) \neq 0$ if and only if $(k_{1},k_{2}) = \pm (j_{1},j_{2})$ .  Similarly, assume that  forcing $S\in L^{\infty}(0,T_{s};H^{s-\gamma})$ is such that $\hat{S}(k_{1},k_{2},k_{3},t) \neq 0$ if and only if $(k_{1},k_{2}) = \pm (j_{1},j_{2})$, for all $t\in[0,T_{s})$. Then there exists $T \in (0,T_{s}]$ and a unique solution $\Th \in L^{\infty}(0,T;H^{s}) \cap L^{2}(0,T;H^{s+\gamma})$ of the initial value problem \eqref{eq:1}--\eqref{eq:2}, and in addition we have that $\hat{\Th}(k_{1},k_{2},k_{3},t) =0$ whenever $k_{2}/k_{1} \neq j_{2}/j_{1}$.
%
\end{corollary}

\begin{proof}[Proof of Corollary~\ref{cor:well-prepared}]
Let $q = j_{2}/j_{1} \in \QQ$. The conditions on the initial data and the force imply in particular that their frequency support likes on the plane $\PP_{q}$. The existence of a unique smooth solution $\Th$, with frequency support lying on $\PP_{q}$ follows directly from Theorem~\ref{thm:well-prepared:local}. But  $(k_{1},k_{2},k_{3}) \in \PP_{q}$ is equivalent to $k_{2} = q k_{1} = j_{2} k_{1} / j_{1}$, so that $\hat{\Th}(k_{1},k_{2},k_{3}) \neq 0$ only if $k_{2}/k_{1} = j_{2}/j_{1}$.
\end{proof}

\subsection{Global existence and uniqueness for well-prepared data} 
\label{sec:well-prepared:global}
In the previous section we have proven that for $\gamma \in (0,1)$, given initial data $\Th_0$ and source term $S$ which are well-prepared, i.e. they have Fourier support on a plane $\PP_q$ for some $q>0$, then there exists a local in time solution $\Th$ of the Cauchy problem \eqref{eq:1}--\eqref{eq:3}, which has the property that its Fourier support also lies on $\PP_q$. In this section we show that for $\gamma \in [1/2,1)$ the local in time solution may be continued for all time, while in the case $\gamma \in (0,1/2)$ the same result holds but under the additional assumption that the initial data is small with respect to $\kappa$. 

\begin{theorem}[{\bf $ \gamma \geq 1/2$: Global existence for large data}] \label{thm:well-prepared:global}
Let $s, \Th_0$ and $S$ be as in the statement of Theorem~\ref{thm:well-prepared:local}. If $\gamma \geq 1/2$, the unique smooth solution $\Th$ of the Cauchy problem  \eqref{eq:FractionalMG:1}--\eqref{eq:FractionalMG:2} is global in time.
\end{theorem}

\begin{proof}[Proof of Theorem~\ref{thm:well-prepared:global}]
 Given the conditions on the initial data and of the source term, and using the properties exhibited in Proposition~\ref{prop:support}, the solution constructed in Theorem~\ref{thm:well-prepared:local} satisfies $\supp(\hat{\Th}) \subset \PP_{q}$. By Lemma~\ref{lemma:M:order:0} we have that $\Ub$ is obtained from $\Th$ by a {\em bounded} Fourier multiplier, and hence by the H\"ormander-Mikhlin theorem we have 
 \begin{align}
 \Vert \Ub \Vert_{W^{s,p}} \leq C \Vert \Th \Vert_{W^{s,p}}
 \end{align} 
 for all $2\leq p < \infty$, and all $s\geq 0$. 
Therefore the regime $\gamma > 1/2$ becomes ``sub-critical'' for such solutions, since the map $\Th\mapsto\Ub$ is bounded in Sobolev spaces. Therefore, energy estimates which are similar to those used to prove the global regularity of the sub-critical SQG equation (cf.~\cite{ConstantinWu99}), may be used with minor modifications to prove Theorem~\ref{thm:well-prepared:global} in the setting $\gamma > 1/2$.
 
The case $\gamma = 1/2$ is slightly more delicate. In the case of the critical (SQG) equation, global regularity in the critical case has only been established recently (cf.~\cite{CaffarelliVasseur10, KiselevNazarovVolberg07} ).
Due to the anisotropic nature of the symbol $\MM$, it turns out that it is slightly easier to see that the DeGiorgi-inspired proof of \cite{CaffarelliVasseur10} also takes care of the $\gamma=1/2$ case of this theorem. The only fact we must verify is that when $\Th \in L^\infty$, and $\supp(\hat \Th) \in \PP_q$, then\footnote{Note that for generic $\Th \in L^\infty$ functions $M[\Th]$ does not belong to $BMO$, but rather to $BMO^{-1}$ (cf.~\cite[Section 4]{FriedlanderVicol11a}).} $\Ub = \MM[\Th] \in BMO$. Once we prove this, the DeGiorgi iteration scheme of \cite{CaffarelliVasseur10}, shows that $\Th(t,\cdot)$ is H\"older continuous for $t>0$, and one may use energy arguments to conclude the proof of Theorem~\ref{thm:well-prepared:global}, similarly to \cite[Appendix]{FriedlanderVicol11c}. 
Lastly, to verify that $\Ub \in BMO$,  define a new operator $\MM_q$ as the Fourier multiplier with symbol $\hat{\MM}(\kk) \varphi_q(\kk)$, where $\varphi_q(\kk)$ is a function which is identically $1$ on $\PP_q$, and vanishes identically at a fixed distance away from $\PP_q$. The advantage is that $\MM_q$ is now a periodic pseudo-differential operator of order $0$, and hence maps $L^\infty$ to periodic $BMO$ (cf.~\cite{McLean}). This is attributed to boundedness  of $\hat{\MM}_q$, inherited from  the symbols $\MM_j$, and follows from the fact that $\mathcal{P}_q$ and its collar neighborhood are not entirely contained in the curved region of the frequency space where $\hat{M}_j$ become unbounded. Lastly one may verify that when $\supp(\hat{\Th}) \subset \PP_q$, then $\MM_q [\Th] = \MM[\Th]$, thereby concluding the proof of the theorem. 
\end{proof}

\begin{theorem}[{\bf $\gamma < 1/2$: Global existence for small data}] \label{thm:well-prepared:global:a}
Let $\gamma, s, \Th_0$ and $S$ be as in the statement of Theorem~\ref{thm:well-prepared:local}.
 If $\gamma  <1/2$,  there exists a sufficiently small constant $\varepsilon>0$, such that if 
 $\Vert\Th_0\Vert_{L^2}^\alpha \Vert \Th_{0} \Vert_{H^{s}}^{1-\alpha} +\Vert\Th_0\Vert_{L^2}^\alpha \Vert S \Vert_{L^\infty(0,\infty;H^{s-\gamma})}^{1-\alpha}\leq \varepsilon$, where $\alpha = 1 - (5/2-\gamma)/s$, then the unique smooth solution $\Th$ of the Cauchy problem  \eqref{eq:FractionalMG:1}--\eqref{eq:FractionalMG:2} is global in time.
\end{theorem}
\begin{proof}[Proof of Theorem~\ref{thm:well-prepared:global:a}]
Since $\MM$ acts as an operator of order $0$ when applied to functions with frequency support on $\PP_q$, from the point of view of energy estimates the situation we are in is exactly the same as for the super-critical (SQG) equation. The proof the theorem follows from the same arguments used in~\cite{CordobaCordoba04, Ju04, Wu04} to show the global well-posedness of small solutions to the super-critical (SQG) equation.
\end{proof}

\appendix
\section{Smooth solutions to a linear equation for~data~with~``thin''~Fourier~support} \label{sec:app:linear}
The goal of this appendix is to prove the existence of smooth solutions to the scalar {\em linear} equation
  \begin{align}
    & \partial_t \theta + v \cdot \nabla \theta + (-\Delta)^\gamma \theta = S \label{eq:app:1}\\
    & \theta|_{t=0} = \theta_0 \label{eq:app:2}
  \end{align}
where the initial datum $\theta_{0}$, the {\em given} divergence-free drift velocity field $v$, and the external source term $S$, all have frequency support in the same plane $\PP_{q}$, i.e.
\begin{align}
\supp(\hat{\theta}_0), \supp(\hat S (t,\cdot)), \supp(\hat v (t,\cdot)) \subset \PP_q \label{eq:app:3}
\end{align}
for all $t\in [0,T]$, for some $q \in \QQ$, $q\neq 0$. The main result is:
\begin{theorem}[{\bf Existence of solutions with support on a plane}]
\label{thm:app:linear}
Let $\gamma \in (0,1)$, $ d\geq 2$, and fix $s > d/2 + 1 - \gamma$. Given $\theta_{0} \in H^{s}(\TT^{d})$, a divergence-free $v\in L^{\infty}(0,T;H^{s}(\TT^{d}))$, and  $S\in L^{\infty}(0,T;H^{s-\gamma}(\TT^{d}))$, such that \eqref{eq:app:3} holds, there exists a unique solution
\begin{align*}
\theta \in L^{\infty}(0,T;H^{s}(\TT^{d})) \cap L^{2}(0,T;H^{s+\gamma}(\TT^{d}))
\end{align*} of the initial value problem \eqref{eq:app:1}--\eqref{eq:app:2}, and we have that
\begin{align*}
\supp(\hat \theta (t,\cdot)) \subset \PP_{q}
\end{align*}
holds for all $t \in [0,T]$.
\end{theorem}
\begin{remark} \label{rem:app}
The main difficulty in proving Theorem~\ref{thm:app:linear} is in designing an iteration scheme which is both suitable  for energy estimates, and preserves the feature that in each iteration step the frequency support of the approximation lies on $\PP_{q}$. In this direction we note that a scheme such that $\partial_{t} \theta_{n+1}$ is given in terms of $v \cdot \nabla \theta_{n}$ automatically preserves the frequency support on $\PP_{q}$ in view of Lemma~\ref{prop:support}, but is not suitable for closing the estimates at the level of Sobolev spaces. On the other hand, if we consider $\partial_{t} \theta^{(n+1)}$ to depend on $v \cdot \nabla \theta_{n+1}$, while energy estimates are now clear, it seems difficult to inductively obtain that $\supp( \hat{\theta}_{n+1}) \subset \PP_{q}$.
\end{remark}
\begin{proof}[Proof of Theorem~\ref{thm:app:linear}]

Since the iteration scheme which is suitable for controlling  the frequency support of the solution is ``loosing'' a derivative, we regularize \eqref{eq:app:1}--\eqref{eq:app:2} with hyper-dissipation as
\begin{align}
&\partial_{t} \theta^{\eps} + v \cdot \nabla \theta^{\eps} + (-\Delta)^{\gamma} \theta^{\eps} - \eps \Delta \theta^{\eps} = S \label{eq:app:1eps}\\
&\theta^{\eps}(0,\cdot) = \theta_{0} \label{eq:app:2eps}
\end{align}
for $\eps \in (0,1]$, and later pass to the limit  $\eps \rightarrow 0$ in order to obtain a solution of the original system. Since $v$ and $S$ are smooth, and $v$ is divergence-free, it follows from our earlier paper~\cite{FriedlanderVicol11a} that there exists a unique global (or as long as $v$ and $S$ permit) smooth solution $\theta^{\eps}$ of \eqref{eq:app:1eps}--\eqref{eq:app:2eps}, with
\begin{align}
\theta^{\eps} \in L^{\infty}(0,T;H^{s}) \cap L^{2}(0,T;H^{s+\gamma}) \cap \eps L^{2}(0,T;H^{s+1}).\label{eq:app:regularity}
\end{align}
and the solution is bounded in the above spaces {\em independently of} $\eps \in (0,1]$.

The advantage of considering the system \eqref{eq:app:1eps}--\eqref{eq:app:2eps} over \eqref{eq:app:1}--\eqref{eq:app:2}, is that for the hyper-dissipative system, we can {\em construct} a smooth solution (as will be shown below), which has frequency support lying in $\PP_{q}$. Since \eqref{eq:app:1eps} is linear, and we work in the smooth category, i.e. $s>d/2+1-\gamma$, the {\em unique} smooth solution of \eqref{eq:app:1eps}--\eqref{eq:app:2eps} satisfies $\supp(\hat{\theta}^{\eps}(t,\cdot)) \subset \PP_{q}$ for all $t\in [0,T)$. We note already that the $\eps$-independent bounds in the regularity class \eqref{eq:app:regularity}, will allow us to pass to a limit $\theta^{\eps} \rightarrow \theta$, as $\eps \rightarrow 0$, and this limiting function will automatically satisfy $\supp(\hat \theta) \subset \PP_{q}$, since the latter space is closed and discrete.

We now proceed to construct a solution $\theta^{\eps}$ of \eqref{eq:app:1eps}--\eqref{eq:app:2eps}, which has the desired frequency support property. We consider the following iterative scheme: the first iterate is given by the solution of
\begin{align}
& \partial_{t} \theta_{1}^{\eps} + (-\Delta)^{\gamma} \theta_{1}^{\eps} - \eps \Delta \theta_{1}^{\eps} = S \label{eq:app:1st:1} \\
& \theta_{1}^{\eps}(0,\cdot) = \theta_{0} \label{eq:app:1st:2}
\end{align}
for all $t\in [0,T)$, while the $n+1^{st}$ iterate is given as the unique solution of
\begin{align}
&\partial_{t} \theta_{n+1}^{\eps} + v \cdot \nabla \theta_{n}^{\eps}  + (-\Delta)^{\gamma} \theta_{n+1}^{\eps}  - \eps \Delta \theta_{n+1}^{\eps} = S\label{eq:app:nth:1}\\
&\theta_{n+1}^{\eps}(0,\cdot) = \theta_{0} \label{eq:app:nth:2}
\end{align}
for all $n \geq 1$. We note that the solutions of \eqref{eq:app:1st:1}--\eqref{eq:app:1st:2} and \eqref{eq:app:nth:1}--\eqref{eq:app:nth:2} respectively, may be written explicitly using the Duhamel formula
\begin{align}
&\theta_{1}^{\eps}(t) = e^{(\eps (-\Delta) +(-\Delta)^{\gamma})t} \theta_{0} + \int_{0}^{t}e^{(\eps (-\Delta) +(-\Delta)^{\gamma}) (t-\tau)} S(\tau) d\tau \label{eq:app:theta:1}\\
&\theta_{n+1}^{\eps}(t) = e^{(\eps (-\Delta) +(-\Delta)^{\gamma})t} \theta_{0} + \int_{0}^{t}e^{(\eps (-\Delta) +(-\Delta)^{\gamma}) (t-\tau)} \left(\nabla \cdot (v\; \theta_{n}^{\eps})(\tau) + S(\tau)\right) d\tau \label{eq:app:theta:n+1}.
\end{align}
Since the operator $\exp\left( (\eps (-\Delta) +(-\Delta)^{\gamma})t\right)$ is given explicitly by the Fourier multiplier with {\em non-zero} symbol $\exp\left( (\eps |\kk| + |\kk|^{\gamma})t\right)$, this operator does not alter the frequency support of the function it acts on. Therefore, it follows directly from Proposition~\ref{prop:support}, our assumptions on the frequency support on $\theta_{0}$ and $S$,  that $\supp( \hat{\theta_{1}^{\eps}}(t,\cdot) ) \subset \PP_{q}$ for all $t\in [0,T)$. We proceed inductively and note that if $\supp( \hat{\theta_{n}^{\eps}}) \subset \PP_{q}$, then by our assumption on the frequency support of $v$ and Proposition~\ref{prop:support} we also have $\supp( \hat{v\,  \theta_{n}^{\eps}} ) \subset \PP_{q}$. Hence, we obtain, as in the case $n=0$, that $\supp( \hat{\theta_{n+1}^{\eps}}(t,\cdot) )\subset \PP_{q}$ for all $t\in[0,T)$, concluding the proof of the induction step. This proves that the frequency support of all the iterates $\theta_{n}^{\eps}$ lies on $\PP_{q}$.

It is left to prove that the sequence $\{ \theta_{n}^{\eps}\}_{n\geq 1}$ converges to a function $\theta^{\eps}$ which lies in the smoothness class \eqref{eq:app:regularity}.
Note that there is no cancellation of the highest order term in the nonlinearity: $\int \Lambda^{s}(v \cdot \nabla \theta_{n}^{\eps}) \Lambda^{s} \theta_{n+1}^{\eps}$. However, since at least for now $\eps>0$ is {\em fixed}, we may use the full smoothing power of the Laplacian. First, note that from \eqref{eq:app:1eps} it follows that for any $t \in (0, T]$ we have
\begin{align}
{\mathcal R}_{1}(t) :=& \sup_{[0,t]} \Vert \Lambda^{s} \theta_{1}^{\eps}(\tau,\cdot) \Vert_{L^{2}}^{2} + \int_{0}^{t} \Vert \Lambda^{s+\gamma} \theta_{1}(\tau,\cdot) \Vert_{L^{2}}^{2} d\tau + \eps \int_{0}^{t} \Vert \Lambda^{s+1} \theta_{1}(\tau,\cdot) \Vert_{L^{2}}^{2} d\tau \notag\\
&\leq \Vert \Lambda^{s} \theta_{0} \Vert_{L^{2}}^{2} + \int_{0}^{t} \Vert \Lambda^{s-\gamma} S(\tau,\cdot) \Vert_{L^{2}}^{2} d\tau. \label{eq:app:R1}
\end{align}
Hence, there exists a time $T_{1} \in (0,T]$ such that ${\mathcal R}_{1}(T_{1}) \leq 2 \Vert \Lambda^{s} \theta_{0} \Vert_{L^{2}}^{2}$, e.g., any $T_{1}$ such that
\begin{align}
T_{1}\leq \frac{\Vert \Lambda^{s} \theta_{0} \Vert_{L^{2}}^{2}}{ \Vert S \Vert_{L^{\infty}(0,T;H^{s-\gamma})}^{2}} \label{eq:app:T1:1}
\end{align} is sufficient. We proceed inductively, and assume that
\begin{align}
{\mathcal R}_{n}(t) :=& \sup_{[0,t]} \Vert \Lambda^{s} \theta_{n}^{\eps}(\tau,\cdot) \Vert_{L^{2}}^{2} + \int_{0}^{t} \Vert \Lambda^{s+\gamma} \theta_{n}(\tau,\cdot) \Vert_{L^{2}}^{2} d\tau + \eps \int_{0}^{t} \Vert \Lambda^{s+1} \theta_{n}(\tau,\cdot) \Vert_{L^{2}}^{2} d\tau
\end{align}
is such that ${\mathcal R}_{n}(T_{1}) \leq 2  \Vert \Lambda^{s} \theta_{0} \Vert_{L^{2}}^{2}$. We now show that if $T_{1}$ is chosen appropriately, in terms of $\eps, \theta_{0}, v$ and $S$, we have ${\mathcal R}_{n+1}(T_{1}) \leq 2 \Vert \Lambda^{s} \theta_{0} \Vert_{L^{2}}^{2}$ too. From \eqref{eq:app:nth:1}, the fact that $\nabla \cdot v = 0$, integration by parts, and $s>d/2+1-\gamma > d/2$ which makes $H^{s}$ an algebra, we obtain
\begin{align}
& \frac{1}{2} \frac{d}{dt} \Vert \Lambda^{s} \theta_{n+1}^{\eps} \Vert_{L^{2}}^{2} + \Vert \Lambda^{s+\gamma} \theta_{n+1}^{\eps} \Vert_{L^{2}}^{2} + \eps \Vert \Lambda^{s+1} \theta_{n+1}^{\eps} \Vert_{L^{2}}^{2}\notag\\
& \qquad \leq \Vert \Lambda^{s} (v \theta_{n}^{\eps}) \Vert_{L^{2}} \Vert  \Lambda^{s+1} \theta_{n+1}^{\eps} \Vert_{L^{2}} + \Vert \Lambda^{s+\gamma} \theta_{n+1}^{\eps} \Vert_{L^{2}} \Vert \Lambda^{s-\gamma} S \Vert_{L^{2}}  \notag\\
& \qquad \leq \frac{1}{2\eps} \Vert \Lambda^{s} v\Vert_{L^{2}}^{2} \Vert \Lambda^{s}\theta_{n}^{\eps} \Vert_{L^{2}}^{2} + \frac{\eps}{2} \Vert  \Lambda^{s+1} \theta_{n+1}^{\eps} \Vert_{L^{2}}^{2} + \frac{1}{2} \Vert \Lambda^{s+\gamma} \theta_{n+1}^{\eps} \Vert_{L^{2}}^{2} +\frac{1}{2}  \Vert \Lambda^{s-\gamma} S \Vert_{L^{2}}^{2}\label{eq:app:R:n+1}
\end{align}
from which it follows that
\begin{align}
{\mathcal R}_{n+1}(T_{1}) &\leq \Vert \Lambda^{s} \theta_{0}\Vert_{L^{2}}^{2} +\frac{1}{\eps} \int_{0}^{T_{1}} \Vert \Lambda^{s} v(\tau,\cdot) \Vert_{L^{2}}^{2} \Vert \Lambda^{s}\theta_{n}^{\eps}(\tau,\cdot) \Vert_{L^{2}}^{2}  d\tau + \int_{0}^{T_{1}} \Vert \Lambda^{s-\gamma} S(\tau,\cdot) \Vert_{L^{2}}^{2} d\tau \notag\\
&\leq \Vert \Lambda^{s} \theta_{0}\Vert_{L^{2}}^{2} + T_{1} \left(  \frac{2}{\eps}  \Vert v\Vert_{L^{\infty}(0,T;H^{s})}^{2}  \Vert \Lambda^{s} \theta_{0} \Vert_{L^{2}}^{2}  + \Vert S\Vert_{L^{\infty}(0,T;H^{s-\gamma})}^{2} \right) .
\end{align}
Hence, it follows from \eqref{eq:app:R1} that if we let
\begin{align}
T_{1} = \frac{\eps \Vert \Lambda^{s} \theta_{0} \Vert_{L^{2}}^{2}}{4  \Vert v\Vert_{L^{\infty}(0,T;H^{s})}^{2}  \Vert \Lambda^{s} \theta_{0} \Vert_{L^{2}}^{2}  +  \eps \Vert S\Vert_{L^{\infty}(0,T;H^{s-\gamma})}^{2} } \label{eq:app:T1:2}
\end{align}
then we have ${\mathcal R}_{n+1}(T_{1}) \leq 2  \Vert \Lambda^{s} \theta_{0} \Vert_{L^{2}}^{2}$.
Since the choice of $T_{1}$ cf.~\eqref{eq:app:T1:2} is independent, of $n$, it is clear that the inductive argument may be carried through, and hence ${\mathcal R}_{n}(T_{1}) \leq 2  \Vert \Lambda^{s} \theta_{0} \Vert_{L^{2}}^{2}$ independently of $n\geq 1$. To pass to a limit in $n$, we consider the difference of two iterates, and note that
\begin{align}
&\partial_{t} (\theta_{n+1}^{\eps} - \theta_{n}^{\eps})  + (-\Delta)^{\gamma}  (\theta_{n+1}^{\eps} - \theta_{n}^{\eps})   - \eps \Delta  (\theta_{n+1}^{\eps} - \theta_{n}^{\eps})   + v\cdot \nabla (\theta_{n}^{\eps} - \theta_{n-1}^{\eps})= 0 \label{eq:app:difference}\\
&  (\theta_{n+1}^{\eps} - \theta_{n}^{\eps}) (0,\cdot) = 0
\end{align}
for all $n\geq 2$.
Similarly to \eqref{eq:app:R:n+1}, it follows from  \eqref{eq:app:difference} that
\begin{align}
\tilde{{\mathcal R}}_{n}(t) :=& \sup_{[0,t]} \Vert \Lambda^{s} ( \theta_{n+1}^{\eps}- \theta_{n}^{\eps}) (\tau,\cdot) \Vert_{L^{2}}^{2} + \int_{0}^{t} \Vert \Lambda^{s+\gamma} (\theta_{n+1}-\theta_{n}^{\eps})(\tau,\cdot) \Vert_{L^{2}}^{2} d\tau \notag\\
&\qquad \qquad \qquad + \eps \int_{0}^{t} \Vert \Lambda^{s+1} (\theta_{n+1}-\theta_{n}^{\eps}) (\tau,\cdot) \Vert_{L^{2}}^{2} d\tau \notag \\
&\leq \frac{2}{\eps} \int_{0}^{t} \Vert \Lambda^{s} v(\tau,\cdot) \Vert_{L^{2}}^{2} \Vert \Lambda^{s} \theta(\tau,\cdot) \Vert_{L^{2}}^{2} d\tau
\end{align}
and therefore
\begin{align}
\tilde{\mathcal R}_{n}(T_{1}) \leq \frac{2 T_{1} }{\eps} \Vert v \Vert_{L^{\infty}(0,T;H^{s})}^{2} \tilde{\mathcal R}_{n-1}(T_{1}) \label{eq:app:diff:2}
\end{align}
for all $n\geq 2$. Thus, due to our choice of $T_{1}$ cf.~\eqref{eq:app:T1:2}, we have that
\begin{align}
T_{1} \leq \frac{\eps}{4  \Vert v \Vert_{L^{\infty}(0,T;H^{s})}^{2}}
\end{align}
and hence $\tilde{\mathcal R}_{n}(T_{1}) \leq \tilde{\mathcal R}(T_{1})/2$, which implies that the sequence $\{ \theta_{n}^{\eps} \}_{n\geq 1}$ is not just bounded, but also a contraction in
\begin{align}
L^{\infty}(0,T_{1};H^{s}) \cap L^{2}(0,T_{1};H^{s+\gamma}) \cap \eps L^{2}(0,T_{1};H^{s+1}). \label{eq:app:regularity:1}
\end{align}
Hence there exists a limiting function $\theta^{\eps}$ in the category \eqref{eq:app:regularity:1}. In addition, since for every $n\geq 1$ we have $\supp(\hat{\theta_{n}^{\eps}}) \subset \PP_{q}$, and the set $\PP_{q}$ is closed (and even discrete), we automatically obtain that $\supp(\hat{\theta^{\eps}}) \subset \PP_{q}$.

To show that $\theta^{\eps}$ may be continued in the category \eqref{eq:app:regularity:1} up to time $T$, we note that $\Vert \Lambda^{s} \theta^{\eps}(T_{1})  \Vert_{L^{2}}^{2} \leq 2 \Vert \Lambda^{s} \theta_{0} \Vert_{L^{2}}^{2}$, and hence repeating the above argument with initial condition $\theta^{\eps}(T_{1})$, we obtain a solution $\theta^{\eps} \in L^{\infty}(0,T_{1}+T_{2};H^{s}) \cap L^{2}(0,T_{1}+T_{2};H^{s+\gamma}) \cap \eps L^{2}(0,T_{1}+T_{2},H^{s+1})$, where
\begin{align}
T_{2} = \frac{2 \eps \Vert \Lambda^{s} \theta_{0} \Vert_{L^{2}}^{2}}{8  \Vert v\Vert_{L^{\infty}(0,T;H^{s})}^{2}  \Vert \Lambda^{s} \theta_{0} \Vert_{L^{2}}^{2}  +  \eps \Vert S\Vert_{L^{\infty}(0,T;H^{s-\gamma})}^{2} }
\end{align}
which is such that $\Vert \Lambda^{s} \theta^{\eps}(T_{1}+T_{2})\Vert_{L^{2}}^{2} \leq 4 \Vert \theta_{0} \Vert_{L^{2}}^{2}$. Since the series
\begin{align}
\sum_{k\geq 0} \frac{2^{k} \eps \Vert \Lambda^{s} \theta_{0} \Vert_{L^{2}}^{2}}{2^{k+2}\Vert v\Vert_{L^{\infty}(0,T;H^{s})}^{2}  \Vert \Lambda^{s} \theta_{0} \Vert_{L^{2}}^{2}  +  \eps \Vert S\Vert_{L^{\infty}(0,T;H^{s-\gamma})}^{2} }
\end{align}
diverges for every fixed $\eps>0$, the above argument may be terminated after finitely many steps, concluding the construction of the solution $\theta^{\eps}$ in the category \eqref{eq:app:regularity}.

In order to conclude the proof of the lemma we need to pass to a limit as $\eps \rightarrow 0$. By construction we have that $\theta^{\eps}$ is uniformly bounded, with respect to $\eps$,  in $L^{\infty}(0,T;H^{s}) \cap L^{2}(0,T;H^{s+\gamma})$, and from \eqref{eq:app:1eps} we obtain that $\partial_{t} \theta^{\eps}$ is uniformly bounded, with respect to $\eps$, in $L^{2}(0,T;H^{s-2+\gamma})$. Thus, by the Aubin-Lions compactness lemma (see for instance~\cite{Temam01}), we obtain a weak limit $\theta \in L^{\infty}(0,T;H^{s}) \cap L^{2}(0,T;H^{s+\gamma})$, so that the convergence $\theta^\eps \rightarrow \theta$ is strong in $L^{2}(0,T;H^{s})$.
Since the evolution is linear and $s$ is large enough, it follows that this limiting function is the unique solution of \eqref{eq:app:1} which lies in $L^{\infty}(0,T;H^{s}) \cap L^{2}(0,T;H^{s+\gamma})$. Lastly, since for every $\eps>0$ we have $\supp(\hat{\theta^{\eps}}) \subset \PP_{q}$, and since $\PP_{q}$ is closed, we obtain that the limiting function also has the desired support property, i.e. $\supp(\hat{\theta}) \subset \PP_{q}$, which concludes the proof of the theorem.
\end{proof}

\subsection*{Acknowledgements} The work of SF is supported in part by the NSF grant DMS-0803268. We are grateful to Pierre Germain for stimulating discussions on an early version of this paper.


\end{document}